\documentclass{article}
\usepackage{amssymb,amscd,amsthm, verbatim,amsmath,color,fancyhdr, mathrsfs}
\usepackage{tikz}
\usepackage{fullpage}
\usepackage{graphicx}
\usepackage{turnstile}
\usepackage{mathtools}
\usepackage{bm}
\usepackage{inconsolata}  
\usepackage{comment}
\usepackage[colorinlistoftodos,bordercolor=orange,backgroundcolor=orange!20,linecolor=orange,textsize=scriptsize]{todonotes}

\usepackage{mathtools} 

\usepackage[colorlinks=true,citecolor=blue]{hyperref}

\usepackage[all]{xy}
\usepackage{float}
\usepackage{bbold}
\usepackage{enumerate}
\usepackage{listings}

\newcommand{\RR}{\mathbb{R}}

\newcommand{\FF}{\mathcal{F}}

\newcommand{\PP}{\mathcal{P}}

\newcommand{\DD}{\mathcal{D}}

\DeclareMathOperator{\sub}{sub}

\newtheorem{thm}{Theorem}[section]
\newtheorem{Q}[thm]{Question}

\newtheorem{lemma}[thm]{Lemma}
\newtheorem{claim}{Claim}
\newtheorem{prop}[thm]{Proposition}
\newtheorem{cor}[thm]{Corollary}
\newtheorem{pro}[thm]{Problem}

\theoremstyle{definition}

\newtheorem{remark}[thm]{Remark}
\newtheorem{eg}[thm]{Example}

\DeclareMathOperator{\conv}{conv}
\DeclareMathOperator{\cone}{cone}
\DeclareMathOperator{\HST}{HST}

\usepackage[normalem]{ulem} 
\usepackage{bm}

\usepackage[T1]{fontenc} 

\usepackage[mathlines]{lineno}

\newcommand*\patchAmsMathEnvironmentForLineno[1]{%
	\expandafter\let\csname old#1\expandafter\endcsname\csname #1\endcsname
	\expandafter\let\csname oldend#1\expandafter\endcsname\csname end#1\endcsname
	\renewenvironment{#1}%
	{\linenomath\csname old#1\endcsname}%
	{\csname oldend#1\endcsname\endlinenomath}}%
\newcommand*\patchBothAmsMathEnvironmentsForLineno[1]{%
	\patchAmsMathEnvironmentForLineno{#1}%
	\patchAmsMathEnvironmentForLineno{#1*}}%
\AtBeginDocument{%
	\patchBothAmsMathEnvironmentsForLineno{equation}%
	\patchBothAmsMathEnvironmentsForLineno{align}%
	\patchBothAmsMathEnvironmentsForLineno{flalign}%
	\patchBothAmsMathEnvironmentsForLineno{alignat}%
	\patchBothAmsMathEnvironmentsForLineno{gather}%
	\patchBothAmsMathEnvironmentsForLineno{multline}%
}
\usepackage{authblk} 
\usepackage{orcidlink}

\title{On an extension problem on the moment curve}

\date{}
\author{Seunghun Lee\footnote{%
		Department of Mathematics, Keimyung University, Daegu, South Korea. Partially supported by the Institute for Basic Science (IBS-R029-C1) and the Department of Mathematical Sciences of Korea Advanced Institute of Science and Technology (BK21). \texttt{seunghun.math@gmail.com}.} 
	\ and Eran Nevo\footnote{%
		Mathematics Research Institute, Universidad de Valladolid, Spain and 
        Einstein Institute of Mathematics, Hebrew University of Jerusalem, Israel, and Harvard CMSA, Cambridge, USA. Partially supported by Israel Science Foundation grants ISF-2480/20 and ISF-687/24. \texttt{nevo@math.huji.ac.il}}%
}

\usepackage[T1]{fontenc}

\usepackage[utf8]{inputenc}
\usepackage{array}

\begin{document}
	\maketitle

\begin{abstract}
We show that for $2\le d\le 4$, every finite geometric simplicial complex $\Delta$ in $\RR^d$ with vertices on the moment curve 
can be extended to a triangulation $T$ of the cyclic polytope $C$ where $\Delta, T$ and $C$ all have the same vertex set. Further, for $d\ge 5$ we construct for every $n\ge d+3$ complexes $\Delta$ on $n$ vertices for which no such triangulations $T$ exist. 

Our result for $d=4$ has the following novel algebraic application, due to a correspondence
by Oppermann and Thomas (JEMS, 2012): 
every maximal rigid object in $\mathcal{O}_{A_n^{2}}$ is cluster tilting, where $\mathcal{O}_{A_n^{\delta}}$ denotes a higher dimensional cluster category introduced by Oppermann and Thomas for $A_n^\delta$, where  $A_n^\delta$ denotes a higher Auslander algebra of linearly oriented type $A$.
\end{abstract}
    
\section{Introduction} \label{sec_intro}
The \textit{moment curve} in $\RR^d$, $\gamma_d=\{(t,t^2, \dots, t^d):\ t\in \RR\}$, has remarkable properties used in applications to combinatorics and discrete geometry, in particular in the study of polytopes: for every finite point set $A\subseteq \gamma_d$ with $|A|=n\ge d+1$, the convex hull of $A$, $\conv(A)$, is the \textit{cyclic polytope} on $n$ vertices of dimension $d$, $C(n,d)$, a simplicial polytope whose face lattice is determined by $n$ and $d$, and which maximizes the face numbers among all $d$-dimensional polytopes on $n$ vertices by McMullen's celebrated upper bound theorem \cite{upper_bound_mcmullen}. See e.g. the textbooks~\cite{triangulations_book, grunbaum_polytope_book, lectures_on_polytopes_book} for more details.     

\smallskip 
We consider the following natural extension problem.

\begin{Q} \label{question_main}
Does \emph{every} geometric simplicial complex in $\RR^d$ on a finite point set $A\subseteq \gamma_d$ extend to a triangulation of $\conv(A)$ without adding new vertices?
\end{Q}

This problem is in fact a combinatorial one, as the question of whether a collection of $d$-simplices on the ground set $A\subseteq \gamma_d$ induces a triangulation of $\conv(A)$ can be phrased in combinatorial terms: intersections correspond to interlacing patterns (see Proposition \ref{prop_overlap}) which are forbidden, and boundary faces are identified by the Gale evenness condition (e.g.~\cite{grunbaum_polytope_book}) and one checks that exactly these $(d-1)$-faces are contained in a unique $d$-simplex.

As we shall see, the answer to Question \ref{question_main} depends on the dimension $d$. Considering a more general situation first, note that for $d=2$ every simplicial complex on a finite point set $A\subseteq \RR^2$ extends to a triangulation of $\conv(A)$, while for $d=3$ there are finite point sets $A\subseteq \RR^3$ in convex general position, and simplicial complexes on them, which do not extend to a triangulation of $\conv(A)$ without adding new vertices. As a simple example consider the Sch\"{o}nhardt's polyhedron~\cite{Schönhardt1928}, see \cite[Lemma 3.6.2, Figure 3.43]{triangulations_book}: let $A$ be the set of six vertices of a triangular prism with its bottom triangle face rotated a little, and let $\Delta$ be the union of the three tetrahedra where their vertex sets correspond to the three square sides of the prism, after being rotated. Then $\Delta$ does not extend to a triangulation of $\conv(A)$ without new vertices.

Perhaps surprisingly, we prove that the answer to Question~\ref{question_main} is Yes iff $d\le 4$. To state it precisely, let us set some terminology first. For an affinely independent set $\sigma \subset \RR^d$ of size at most $d+1$, we will frequently refer to the geometric simplex $\conv(\sigma)$ as $\sigma$ when it is appropriate. 
We say that two simplices $\sigma_1$ and $\sigma_2$ in $\RR^d$ \textit{overlap} if the intersection of the simplices is not a face of each of them, namely
\[\conv(\sigma_1) \cap \conv(\sigma_2) \supsetneq \conv (\sigma_1\cap \sigma_2).\]
In this definition, non-overlapping simplices may intersect, but only in a common face. The non-strict containment above 
always holds, and when it holds with equality for every pair of simplices of a given simplicial complex we call it a geometric embedding of the complex. 

For convenience in the proof we use a different notation for dimension in our main result:
    \begin{thm} \label{thm_main}
(i) For every $D\le 4$ and every finite collection $\FF$ of pairwise non-overlapping simplices in $\RR^D$ on $A\subseteq \gamma_D$, $\FF$ can be extended into a triangulation $T$ of the cyclic polytope $\conv(A)$ such that $A$ is exactly the vertex set of $T$. 

(ii) For every $D\ge 5$ and $n \geq D+3$, there exists a collection of pairwise non-overlapping $D$-simplices in $\RR^D$ exactly on an $n$-vertex set $A\subseteq \gamma_D$  which cannot be extended into a triangulation of $\conv(A)$ without adding new vertices.
    \end{thm}
The additional vertices required to extend the collection on $A$ given by Theorem \ref{thm_main}(ii) into a triangulation of $\conv(A)$ are not on $\gamma_D$ as they should be in $\conv(A)$.

In order to prove Theorem~\ref{thm_main}(i) we will use the known facts that for $d=2,3$ the higher Stasheff-Tamari poset $\HST(n,d)$ is a lattice; see~\cite{HUANG19727} and~\cite[Thm.3.6]{edelman_cyclic_triangulation_envelope} respectively and Theorem~\ref{lemma_lattice}.  In order to prove Theorem~\ref{thm_main}(ii) we start with an example of Rambau~\cite{rambau} with $D=5$ and $n=8$, and then lift it to all $D\geq 5$ and $n\ge D+3$.
For $D\ge 6$ even and $n=D+3$, different 
non-extendable 
examples where  
constructed by Oppermann and Thomas~\cite{opperman_thomas}. 

\smallskip

Theorem \ref{thm_main}(i) has an algebraic implication. In the seminal paper by Oppermann and Thomas \cite{opperman_thomas}, they investigated connections between $A_n^\delta$, the $(\delta-1)$-st higher Auslander algebras of linearly oriented $A_n$ introduced by Iyama in the context of higher Auslander-Reiten theory (see \cite{iyama, iyama_oppermann}), and triangulations of cyclic $2\delta$-polytopes. 
For $\mathcal{O}_{A_n^\delta}$, a cluster category introduced in~\cite{opperman_thomas} for $A_n^\delta$,
they showed:
\begin{thm}[{\cite[Theorem 1.2]{opperman_thomas}}]\label{thm_oppermann_thomas}
    There is a bijection between internal $\delta$-simplices of $C(n+2\delta+1, 2\delta)$ and indecomposable objects in $\mathcal{O}_{A_n^\delta}$. It induces a bijection between triangulations of $C(n+2\delta+1, 2\delta)$ and basic cluster tilting objects in $\mathcal{O}_{A_n^\delta}$.
\end{thm}
In this correspondence, a family of non-overlapping (internal) $\delta$-simplices corresponds to a ``rigid object'' in $\mathcal{O}_{A_n^\delta}$ by considering its  collection of indecomposable direct summands, where each of the direct summands corresponds to a simplex in the family. In \cite[Section 8]{opperman_thomas}, they mentioned that for the classical $\delta=1$ case every maximal rigid object in $\mathcal{O}_{A_n^\delta}$ is cluster tilting (for a general version of this fact, see \cite[Theorem 2.6]{calabi-yau-category-rigid_Zhou_Zhu} and \cite[Theorem II.1.8]{calabi-yau-category-clustuer_Buan_Iyama_Reiten_Scott}), while they provided a counterexample to the analogous assertion for $\delta \geq 3$. For $\delta=2$, their computer experiments did not detect such counterexamples.
Theorem \ref{thm_main}(i) for $D=4$, combined with Theorem \ref{thm_oppermann_thomas}, fills in this gap by settling the $\delta=2$ case:
\begin{cor}\label{cor:algebraic}
    Every maximal rigid object in $\mathcal{O}_{A_n^\delta}$ is cluster tilting when $\delta=2$.
\end{cor}

Thus, our main result exhibits a combinatorial approach towards a purely algebraic problem. It may be interesting to go in the other direction (in the spirit of \cite{williams_stasheff_tamari_advances_lattice}), namely to give a purely algebraic proof of Theorem \ref{thm_main}.

\medskip

\textbf{Previous studies.} The combinatorics of triangulations of the cyclic polytope $C(n,d)$ has been extensively studied. 
Edelman and Reiner \cite{edelman_cyclic_triangulation_envelope} defined two partial orders on this set of triangulations; one is given by bistellar flips and the other 
by comparing the ``heights of the liftings in $\RR^{d+1}$'' for those triangulations. These two orders define the  so-called \textit{(first and second) higher Stasheff-Tamari poset}. They showed that these two orders are equal for $d=2,3$. Since then there have been active research on this topic \cite{rambau,edelman_rambau_reiner_subdivision_lattice_countereg, snug_thomas, opperman_thomas, williams_stasheff_tamari_advances_lattice, williams2022representation}. Recently the two higher Stasheff-Tamari orders were proved to be the same for every dimension $d$ by Williams 	\cite{williams_twoordersequal}.

The combinatorial criterion for having overlapping simplices on the moment curve (see Proposition \ref{prop_overlap}) is well-known~\cite{breen1973primitive} and is one of the ways of defining the \textit{alternating oriented matroids} \cite[Section~9.4]{om_book}. It is also closely related to recent works on geometric hypergraph colorings and their representations where a similar combinatorial property was introduced and has been studied \cite{ABAB_stabbed_pseudodisk, ABA, dual_ABAB2024, complexity_ABAB}: A hypergraph $H$ on a totally ordered vertex set is $(AB)^{l/2}$\textit{-free} if there is no subsequence of vertices of length $l$ following the order such that the odd-numbered vertices are in $A\setminus B$ and the even numbered vertices are in $B\setminus A$ for some hyperedges $A$ and $B$. 
For a hypergraph $H$ consisting of $d$-simplices on $\gamma_{2d}$ with the natural order, we have an equivalence by Proposition \ref{prop_overlap} that $H$ is $(AB)^{d+1}$-free if and only if any pair of two $d$-simplices in $H$ do not overlap.    
    
\medskip

\textbf{Outline.} 
In Section~\ref{sec_prelim} we review the higher Stasheff-Tamari poset $\HST(n,d)$, which is a partial order on the triangulations of the cyclic polytope $C(n,d)$, and analogously define a partial order on simplices in $\RR^d$ by lifting to the moment curve one dimension higher.
In Section~\ref{sec_lattice} we first prove Theorem~\ref{thm_main}(i) using the fact from \cite{edelman_cyclic_triangulation_envelope}, Lemma~\ref{lemma_lattice}, that $\HST(n,d)$ is a
lattice for $d=2,3$, modulo the existence of triangulations in $\HST(n,2)$ and $\HST(n,3)$ with some special properties, Theorems \ref{thm_level_d=2} and \ref{thm_level_d=3}, that are defined from a given collection of pairwise non-overlapping simplices on the moment curve -- the proof of their existence 
is shown after that. Theorem~\ref{thm_main}(ii) is proved in Section~\ref{sec_countereg} by means of constructions. In Section~\ref{sec:final} we discuss complexity aspects of this work and end with related open problems.

\section{Ordering simplices and triangulations by height} \label{sec_prelim}
In this preliminary section, we introduce various relations and partial orders on simplices and triangulations on the moment curve defined by the height of their ``liftings'' and discuss several properties of them, old and new. 
In particular, the partial order on triangulations of $C(n,d)$ was originally introduced by Edelman and Reiner \cite{edelman_cyclic_triangulation_envelope} as the \textit{second higher Stasheff-Tamari order}, see Subsection \ref{subsec_height_triang_simplex} for details. We analogously define the ``above-below'' relation $<_{d+1}$ between simplices, and a partial order $\preceq_{d+1}$ as the reflexive and transitive closure of $<_{d+1}$. The order $\preceq_{d+1}$ is used in the proof of Theorem \ref{thm_main}(i).

\subsection{Ordering simplices by height} \label{subsec_height_order}
We first discuss properties of the moment curve $\gamma_d$ which is useful in our discussion.

Every set of at most $d+1$ distinct points on $\gamma_d$ is affinely independent and every $n$ points on $\gamma_d$ are in convex position. 
The combinatorial properties, to be discussed below, of finite point sets on $\gamma_d$ and their convex hulls 
depend solely on the order in which the points appear on $\gamma_d$ rather than the actual coordinates of the points. Hence when we are given $n$ points $\gamma_d(t_1), \gamma_d(t_2), \dots, \gamma_d(t_n)$ with $t_1<t_2<\cdots <t_n$, 
we will frequently abuse notation and refer to the point $\gamma_d(t_i)$ by its index $i$; so the whole point set is identified with $[n]:=\{1,2,\dots, n\}$. Also, we may consider \textit{the} cyclic polytope  $C(n,d)=\conv\{\gamma_d(i): i\in[n]\}$ on the vertex set $[n]$. 
Following the convention from the Introduction, for $k\leq d$ a geometric $k$-simplex $\conv(\sigma)$ on a vertex set $\sigma \subseteq \gamma_d$ of size $k+1$ will be frequently identified with the subset $\sigma$. 
Note that when we consider two simplices $\sigma$ and $\tau$ on $\gamma_d$, their sizes may differ.

\smallskip

For positive integers $k$ and $a<b$, let \[\DD(k,[a,b])=\{\{i_1<i_2<\cdots<i_{2k-1}<i_{2k}\} \subseteq [a,b] : i_{2j}-i_{2j-1}=1 \textrm{ for every $j\in [k]$}   \},\]
where $[a,b]:=\{a,a+1, \dots, b-1,b\}$.
The following facts are known (e.g. \cite[Lemma 2.3]{edelman_cyclic_triangulation_envelope}).

\begin{prop} \label{prop_gale}
	The set of the facets in the upper envelope of $C(n,d)$ is given by
	\begin{itemize}
		\item $\{\{n\}\}*\DD((d-1)/2,[n-1])$ when $d$ is odd, and
		\item $\{\{1,n\}\} * \DD(d/2-1, [2,n-1])$ when $d$ is even,
	\end{itemize}
	and the set of the facets in the lower envelope of $C(n,d)$ is given by
	\begin{itemize}
		\item $\{\{1\}\}*\DD((d-1)/2,[2,n])$ when $d$ is odd, and
		\item $\DD(d/2, [n])$ when $d$ is even.
	\end{itemize}
This in particular gives the whole list of facets of $C(n,d)$. Here, $*$ is the join operation for families of sets on disjoint ground sets; $\mathcal{A}*\mathcal{B}:=\{A\cup B: A\in \mathcal{A},\, B\in \mathcal{B}\}$.
\end{prop}

For subsets $\sigma$ and $\tau$ of $[n]$, we say that $\sigma$ and $\tau$ are \textit{$k$-interlacing} if there is a sequence $v_1< v_2< \cdots< v_k$ of elements of $\sigma \cup \tau$ which satisfies one of the following alternating conditions:
\begin{align}
&v_1 \in \sigma, v_2 \in \tau, v_3 \in \sigma \dots\textrm{ (beginning with an element of $\sigma$), or} \tag{$\sigma$}\label{interlacing_type1}\\
&v_1 \in \tau, v_2 \in \sigma, v_3 \in \tau \dots \textrm{ (beginning with an element of $\tau$)}. \tag{$\tau$}\label{interlacing_type2}
\end{align}
The following fact gives a combinatorial criterion for the overlapping condition on the moment curve and 
is well-known~\cite[Theorem]{breen1973primitive}; see also  \cite[Lemma 2.5]{lee_nevo2023colorings}).

\begin{prop}\label{prop_overlap}
Simplices $\sigma$ and $\tau$ on $\gamma_d$ overlap in $\RR^d$ if and only if 
they are $(d+2)$-interlacing.
\end{prop}

For $k\leq d+1$ and a simplex $\sigma=\{\gamma_d(t_1), \dots, \gamma_d(t_k)\}$ on $\gamma_d$, we define the \textit{height function} of $\sigma$, $h_\sigma: \conv(\sigma) \to \RR$,  by setting $h_\sigma(p)$, for each $p \in \conv(\sigma)$, to be the last coordinate of the point in $\conv(\hat{\sigma})$ corresponding to $p$ via projection, where $\hat{\sigma}=\{\gamma_{d+1}(t_1), \dots, \gamma_{d+1}(t_k)\}$. We call $\hat{\sigma}$ the \textit{lifting of} $\sigma$.

For simplices $\sigma$ and $\tau$ on $\gamma_d$, we want to compare their heights (or more precisely, the heights of the liftings $\hat{\sigma}$ and $\hat{\tau}$) in $\RR^{d+1}$. 
If $\sigma$ and $\tau$ do not overlap in $\RR^d$, then  $h_\sigma$ and $h_\tau$ have the same value at the common domain. Thus we only compare the heights when $\sigma$ and $\tau$ overlap in $\RR^d$, or equivalently by Proposition \ref{prop_overlap}, when they are $(d+2)$-interlacing.

First, we consider the following observation.

\begin{prop} \label{prop_interlacing_implies_strict_less_pt}
Suppose that two simplices $\sigma$ and $\tau$ on $\gamma_d$ are $(d+2)$-interlacing by
\begin{itemize}
    \item satisfying \eqref{interlacing_type1} when $d$ is even, or 

    \item satisfying \eqref{interlacing_type2} when $d$ is odd.
\end{itemize}
Then there is a point $p$ at the common domain of $h_\sigma$ and $h_\tau$ such that $h_\sigma(p)<h_\tau(p)$.
\end{prop}
\begin{proof}
 When $d$ is even, we can find subsets $\sigma_0\subseteq \sigma$ and $\tau_0 \subseteq \tau$ of size $d/2+1$ which are $(d+2)$-interlacing with an element of $\sigma_0$ comes first. Let $\hat\sigma_0$ and $\hat\tau_0$ be the liftings of $\sigma_0$ and $\tau_0$ on $\gamma_{d+1}$. By Proposition \ref{prop_gale}, in the cyclic polytope $P$ on the vertex set $\hat\sigma_0 \cup \hat\tau_0$, a full simplex in $\mathbb{R}^{d+1}$, we can easily see
that
\begin{enumerate}[(a)]
	\item $\hat\sigma_0$ is at the lower envelope but not at the upper envelope of $P$,
	
	\item $\hat\tau_0$ is at the upper envelope but not at the lower envelope of $P$, and
	
	\item every proper face of $\hat\sigma_0$ or $\hat\tau_0$ belongs to both the lower and upper envelope of $P$.
\end{enumerate}
By the general position property of $\gamma_d$, there is a unique point $p$ at the intersection $\conv(\sigma_0)\cap \conv(\tau_0)$ and $p$ lies at the interior of both $\conv(\sigma_0)$ and $\conv(\tau_0)$. Therefore, (a), (b) and (c) imply that  $h_{\sigma}(p)=h_{\sigma_0}(p)<h_{\tau_0}(p)=h_{\tau}(p)$.

When $d$ is odd, we can similarly argue after finding subsets $\sigma_0\subseteq \sigma$ and $\tau_0 \subseteq \tau$ of size $(d+1)/2$ and $(d+1)/2+1$ respectively, which are $(d+2)$-interlacing. These also satisfy Conditions (a), (b) and (c) by Proposition \ref{prop_gale}, so the desired conclusion holds similarly.
\end{proof}

We now define a relation $<_{d+1}$ on the set of simplices on $\gamma_d$. For two simplices $\sigma, \tau \subseteq \gamma_d$, we have $\sigma <_{d+1} \tau$ if and only if $\sigma$ and $\tau$ overlap in $\RR^d$ and $h_\sigma \leq h_\tau$ in the common domain. The following combinatorial criterion for $<_{d+1}$ will be useful.

\begin{prop} \label{prop_height_interlacing}
For simplices $\sigma$ and $\tau$ on $\gamma_d$, $\sigma <_{d+1} \tau$ if and only if $\sigma$ and $\tau$ are $(d+2)$-interlacing by
\begin{itemize}
    \item  satisfying \eqref{interlacing_type1} but not \eqref{interlacing_type2} when $d$ is even, and

    \item satisfying \eqref{interlacing_type2} but not \eqref{interlacing_type1} when $d$ is odd.
\end{itemize}
\end{prop}

\begin{proof}
For the if implication, suppose the interlacing condition in the statement holds. Then $\sigma$ and $\tau$ are $(d+2)$-interlacing but not $(d+3)$-interlacing, thus 
\begin{itemize}
    \item[(i)] $\sigma$ and $\tau$ do overlap in $\RR^d$, but 

    \item[(ii)] the liftings $\hat\sigma$ and $\hat\tau$ do not overlap in $\RR^{d+1}$;
\end{itemize}
both by Proposition \ref{prop_overlap}. 

Condition (i) implies that for the intersection $D=\conv(\sigma)\cap \conv(\tau) \subseteq \mathbb{R}^d$, we have $D\setminus \conv(\sigma\cap \tau) \ne \emptyset$. Note that $D$ is a convex polytope in $\RR^d$. Then again (i) implies that $\sigma$ and $\tau$ are distinct, so $\sigma \cap \tau$ is a proper face of $\sigma$ or $\tau$.
WLOG, we assume $\sigma \cap \tau$ is a proper face of $\sigma$ (for $\tau$ it can be argued similarly). Let $H$ be a hyperplane in $\mathbb{R}^d$ such that $H\cap \conv(\sigma)=\conv(\sigma\cap \tau)$ and $\conv(\sigma) \subseteq H^+$ for some closed half-space $H^+$ bounded by $H$. Then $H \cap D=H \cap \conv(\sigma) \cap \conv(\tau)=\conv(\sigma\cap \tau)$ and $D\subseteq H^+$. 
This implies that $\conv(\sigma \cap \tau)$ is a proper face of $D$, so $D \setminus \conv(\sigma \cap \tau)$ is again convex.

Further, Condition (ii) implies that $h_\sigma=h_\tau$ only on (the proper face) $\conv(\sigma \cap \tau) \subseteq \RR^d$; otherwise there is a point $p \in D\setminus \conv(\sigma \cap \tau)$ with $h_\sigma(p)=h_\tau(p)$ which means $\hat \sigma$ and $\hat \tau$ have an intersection outside $\conv(\hat\sigma \cap \hat\tau)$, and this leads to a contradiction with (ii).  Since $D \setminus \conv(\sigma \cap \tau)$ is convex, we have either $h_\sigma(p)< h_\tau(p)$ for all $p \in D\setminus \conv(\sigma \cap \tau)$ or $h_\sigma(p)>h_\tau(p)$ for all $p \in D\setminus \conv(\sigma \cap \tau)$. However, only the former case is possible by the interlacing condition and Proposition \ref{prop_interlacing_implies_strict_less_pt}. This shows the if implication.

Now suppose that $\sigma <_{d+1} \tau$. Since $\sigma$ and $\tau$ overlap in $\RR^d$, they are $(d+2)$-interlacing. If the interlacing pattern is different from the one given in the assertion, then by Proposition \ref{prop_interlacing_implies_strict_less_pt}, we can find a point $p$ at the common domain such that $h_\sigma(p)>h_\tau(p)$, a contradiction. This completes the proof.
\end{proof}
 Propositions \ref{prop_overlap} and \ref{prop_height_interlacing} imply the following corollary that is useful later on, see Figure \ref{fig:exclusive_cases_overlap}. Note that we might have $\sigma=\tau$ in Case (A).

 \begin{figure}[ht]
			\centering
			\includegraphics[width=0.8\textwidth]{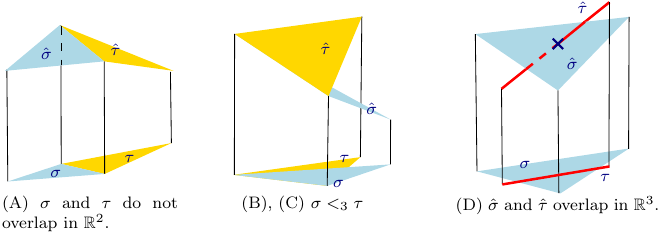}
			\caption{Illustration of the mutually exclusive cases (A)-(D).}\label{fig:exclusive_cases_overlap}
\end{figure}

 \begin{cor} \label{cor_cases}
 Given a pair of simplices $\sigma$ and $\tau$ on $\gamma_d$, one of the following four mutually exclusive cases occurs:
 \begin{enumerate}[(A)]
    \item $\sigma$ and $\tau$ do not overlap in $\RR^d$, or equivalently, $\sigma$ and $\tau$ are not $(d+2)$-interlacing.
    \item $\sigma <_{d+1} \tau$, or equivalently, $\sigma$ and $\tau$ are $(d+2)$-interlacing by satisfying the combinatorial condition of Proposition \ref{prop_height_interlacing}.
    \item $\sigma >_{d+1} \tau$, or equivalently, $\sigma$ and $\tau$ are $(d+2)$-interlacing by satisfying the combinatorial condition of Proposition \ref{prop_height_interlacing} with the roles of $\sigma$ and $\tau$ interchanged.
    \item the liftings  $\hat\sigma$ and $\hat\tau$ overlap in $\RR^{d+1}$, or equivalently, $\sigma$ and $\tau$ are $(d+3)$-interlacing.
\end{enumerate}
    \end{cor}
\begin{proof}
    It is clear that all four cases (A)-(D) are mutually exclusive by their combinatorial 
    interpretation in terms of interlacing patterns. To show that (A)-(D) cover all cases, it is enough to show that (D) holds if neither (A), (B), nor (C) holds. Equivalently, we show that when $\sigma$ and $\tau$ on $\gamma_d$ overlap in $\mathbb{R}^d$, if there are points $p$ and $q$ in the common domain of $h_\sigma$ and $h_\tau$ such that $h_\sigma(p)<h_\tau(p)$ and $h_\sigma(q)>h_\tau(q)$, then the liftings $\hat\sigma$ and $\hat \tau$ overlap in $\mathbb{R}^{d+1}$. In fact, by again setting $D=\conv(\sigma)\cap \conv(\tau)$, we can see that $p, q \in D\setminus \conv(\sigma \cap \tau)$.
    As in the proof of Proposition \ref{prop_interlacing_implies_strict_less_pt}, $D\setminus \conv(\sigma \cap \tau)$ is convex, so along the segment between 
    $p$ and $q$ there exists a point $z \in D\setminus \conv(\sigma \cap \tau)$ such that $h_\sigma(z)=h_\tau(z)$. This is where the liftings $\hat\sigma$ and $\hat \tau$ intersect outside $\conv(\hat \sigma \cap \hat \tau)$, which means that $\hat \sigma$ and $\hat \tau$ overlap in $\mathbb{R}^{d+1}$.
\end{proof} 
\begin{remark}
As conditions (A)-(D) have combinatorial characterizations, a purely combinatorial proof of Corollary \ref{cor_cases} can be given, left to the reader.
\end{remark}
Now, we define another relation $\preceq_{d+1}$ from $<_{d+1}$: given simplices $\sigma$ and $\tau$ on $\gamma_d$, $\sigma \preceq_{d+1} \tau$ if 
\begin{itemize}
\item $\sigma=\tau$, or

\item there is a positive integer $k$ and a sequence of simplices $\sigma=\sigma_0<_{d+1} \sigma_1<_{d+1} \cdots <_{d+1} \sigma_k=\tau$  on $\gamma_d$.
\end{itemize}

\begin{lemma} \label{lemma_linear_extension}
The relation $\preceq_{d+1}$ is a partial order on the simplices on $\gamma_d$.
\end{lemma}
\begin{proof}
That $\preceq_{d+1}$ is 
reflexive and transitive is clear from the definition. We show it is also   
antisymmetric. 
For this, it is enough to show that there is no sequence
$\sigma_0<_{d+1} \sigma_1<_{d+1} \cdots <_{d+1} \sigma_{k-1}<_{d+1}\sigma_k=\sigma_0$ 
for some positive integer $k$ where $\sigma_0, \dots, \sigma_{k-1}$ are all distinct. 

Suppose for contradiction that such a sequence exists for some $k\geq 1$. 
For every $i\in [k]\cup \{0\}$ define the interval $J_i=[\min(\sigma_i), \max(\sigma_i)]$. 
Now we use induction on the dimension $d$. 
For $d=1$, by the condition $\sigma_{i-1}<_{d+1}\sigma_i$ for every $i \in [k]$ and since each $\sigma_i$ is a point or an edge, $J_{i-1} \subsetneq J_{i}$ by Proposition \ref{prop_height_interlacing}. This leads to a contradiction.

For $d=2$, again the condition $\sigma_{i-1}<_{d+1}\sigma_i$ for every $i \in [k]$ implies that $\min(\sigma_{i-1}) \leq \min(\sigma_i)$ and $\max(\sigma_{i-1}) \leq \max(\sigma_i)$ for every $i\in [k]$ by Proposition \ref{prop_height_interlacing}. To have $\sigma_0=\sigma_k$, we must have $J_i=J_{i'}$ for all $i,i'\in[k]$. In particular, all $\sigma_i$ must be triangles, otherwise they do not have the desired interlacing pattern. Let $m_i$ be the remaining point in the singleton $\sigma_i\setminus \{\min(\sigma_i), \max(\sigma_i)\}$. By Proposition \ref{prop_height_interlacing} again, we have $m_{i-1}>m_i$ for every $i\in [k]$; a contradiction, which proves the base case.

Now suppose $d>2$. By Proposition \ref{prop_height_interlacing} and a similar reasoning, we have that when $d$ is even $\min(\sigma_{i-1}) \leq \min(\sigma_i)$ and $\max(\sigma_{i-1}) \leq \max(\sigma_i)$ for every $i\in [k]$, 
and when $d$ is odd
$\min(\sigma_i) \leq \min(\sigma_{i-1})\leq \max(\sigma_{i-1}) \leq \max(\sigma_i)$ for every $i\in [k]$. By the assumption $\sigma_0=\sigma_k$, this implies $J_i=J_{i'}$ for every $i,i'\in [k]$. Let $\sigma_i'=\sigma_i\setminus\{\max(\sigma_i), \min(\sigma_i)\}$. The $\sigma_i'$ are non-empty as $\sigma_i'$ and $\sigma_{i-1}'$ are $d$-interlacing and $d\geq 3$. 
We now consider $\sigma_i'$ as a simplex on $\gamma_{d-2}$ via projection. By Proposition \ref{prop_height_interlacing} again, we have
$\sigma_k'<_{d-1} \sigma_{k-1}'<_{d-1} \cdots <_{d-1} \sigma_1'<_{d-1}\sigma_0'$, which gives a contradiction by the induction hypothesis.
\end{proof}

In particular, we never have $\tau <_{d+1} \sigma$ when $\sigma \preceq_{d+1} \tau$.
It is easy to find examples of non-overlapping (in $\RR^d$) $\sigma, \tau \subseteq \gamma_d$ with $\sigma \preceq_{d+1} \tau$. It is also not very difficult to find $\sigma, \tau \subseteq \gamma_d$ with $\sigma \preceq_{d+1} \tau$ and that the liftings $\hat\sigma$ and $\hat \tau$ overlap in $\RR^{d+1}$:

\begin{eg}\label{eg:height_order_overlap}
Consider the following simplices on $\gamma_2$: $\sigma=\{1,4,6\}$, $e=\{2,7\}$ and $\tau=\{3,5,8\}$. Then $\sigma<_3 e <_3 \tau$ (so in particular $\sigma \preceq_3 \tau$) but $\sigma$ and $\tau$ are 6-interlacing.
\end{eg}

All these suggest that there is subtle difference between the ``above-below'' relation $<_{d+1}$ and the order $\preceq_{d+1}$. While we mostly use $<_{d+1}$ to compare heights between two simplices, $\preceq_{d+1}$ is helpful when we give a total ordering of given simplices consistent with $<_{d+1}$ as in the proof of Theorem \ref{thm_main}(i) at Subsection \ref{subsec_proof_main} and in the proof of Proposition \ref{prop_only_LR} at Subsection \ref{subsec_combinatorial_lemma}.

\subsection{Height of a triangulation and Higher Stasheff-Tamari posets} \label{subsec_height_triang_simplex}

A \textit{triangulation} of a $d$-polytope $P\subseteq \RR^d$ is a geometric simplicial complex whose realization space is exactly the polytope $P$.
Let $S(n,d)$ be the set of all triangulations on vertex set $[n]$ of the cyclic polytope $C(n,d)$. The \textit{lifting of }$T \in S(n,d)$ is the family of the liftings of all simplices of $T$, so it is also a geometric simplicial complex whose realization space is a section inside the cyclic polytope $C(n,d+1)$.

For a triangulation $T \in S(n,d)$, we can similarly define the height function $h_T:C(n,d) \to \RR$ as the union of $h_\sigma: \conv(\sigma) \to \RR$ for every $\sigma \in T$, that is, if a point $p \in \conv(\sigma)$, we let $h_T(p)=h_\sigma(p)$. Then $h_T$ is well-defined for every $T\in S(n,d)$.

This enables us to compare heights between a triangulation and a simplex, or between two triangulations. For a triangulation $T \in S(n,d)$ and a simplex $\sigma \subseteq \gamma_d$, we say $T \leq_{d+1} \sigma$ ($\sigma \leq_{d+1} T$ resp.) if $h_T(p) \leq h_\sigma(p)$ ($h_\sigma(p)\leq h_T(p)$ resp.) for every point $p \in \conv(\sigma)$.

The following lemma follows easily from the definitions and our earlier observations. A similar statement holds for $\sigma \geq_{d+1} T$.

\begin{lemma} \label{lemma_triangulation_simplex_simple}
    For a triangulation $T \in S(n,d)$ and a simplex $\sigma \subseteq \gamma_d$, $\sigma \leq_{d+1} T$ if and only if for every $\tau \in T$ either $\sigma$ and $\tau$ do not overlap in $\RR^d$ or $\sigma <_{d+1} \tau$.
\end{lemma}
\begin{proof}
    By definition, $\sigma \leq_{d+1} T$ if and only if for every $\tau \in T$ whenever $\sigma$ intersects $\tau$ we have $h_\sigma\leq h_\tau$ in the common domain. 
    The latter holds iff either  $\sigma$ and $\tau$ do not overlap in $\RR^d$ or they do overlap and additionally $\sigma <_{d+1} \tau$.
\end{proof}

The following lemma is a reformulation of Lemma \ref{lemma_triangulation_simplex_simple} which will be useful later. It generalizes \cite[Propositions 3.2 and 4.1]{edelman_cyclic_triangulation_envelope}.

\begin{lemma} \label{lemma_triangulation_simplex}
For every triangulation $T \in S(n,d)$ and every simplex $\sigma \subseteq \gamma_d$ there holds: $\sigma \leq_{d+1} T$ (resp. $T\leq_{d+1} \sigma$)  if and only if there are no $\sigma' \subseteq \sigma$ and $\tau \in T$ satisfying $\tau<_{d+1}\sigma'$ with $|\sigma'|=\lceil d/2 \rceil +1$ and $|\tau|=\lfloor d/2 \rfloor+1$
(resp. satisfying $\sigma'<_{d+1}\tau$ with $|\tau|=\lceil d/2 \rceil +1$ and $|\sigma'|=\lfloor d/2 \rfloor+1$).
\end{lemma}
\begin{proof}
We prove the statement for $\sigma \leq_{d+1} T$, the statement for $T\leq_{d+1} \sigma$ is proved analogously.
    Note first that the conditions in the right-hand side of Lemma \ref{lemma_triangulation_simplex_simple} correspond to Conditions (A) and (B) in Corollary \ref{cor_cases} with the consistent notation. So Lemma \ref{lemma_triangulation_simplex_simple} implies that $\sigma \leq_{d+1} T$ if and only if we never have Cases (C) or (D) in Corollary \ref{cor_cases} for the simplex $\sigma$ and any $\tau \in T$, again with the consistent notation. Concerning interlacing patterns, this condition exactly corresponds to the condition at the "right-hand side" of the lemma.
\end{proof}

For two triangulations $T,T'\in S(n,d)$, say $T \leq_{d+1} T'$ if $h_T(p) \leq h_{T'}(p)$ for every point $p\in C(n,d)$. This relation $\leq_{d+1}$ is a partial order on $S(n,d)$, and was introduced by Edelman and Reiner  \cite{edelman_cyclic_triangulation_envelope} as the \textit{second higher Stasheff-Tamari order}.
This order (which coincides with their  \textit{first} higher Stasheff-Tamari order by \cite{williams_twoordersequal})
defines the \textit{higher Stasheff-Tamari poset} $\HST(n,d)=(S(n,d), \leq_{d+1})$. 

The following result by Rambau \cite[Theorem 1.1]{rambau} will be also useful.

\begin{thm} \label{lemma_maximal_chain}
	There is a surjective map $\Psi$ from the set of maximal chains of $\HST(n,d)$ to $S(n,d+1)$ such that 
    for a maximal chain $\mathcal{C}:T_1 \leq_{d+1} T_2 \leq_{d+1} \cdots \leq_{d+1} T_k$ in $\HST(n,d)$, the $d$-skeleton of $\Psi(\mathcal{C})$ is exactly the union of the liftings of all $T_i$ (which in turn determines  $\Psi(\mathcal{C})$).
\end{thm}

 We introduced various relations throughout this section: between simplices, between a simplex and a triangulation, and between triangulations respectively;
however, they are easily distinguishable by context.

\subsection{Submersion set of a triangulation and lattice property}
\label{subsec_submersion}

For a triangulation $T \in S(n,d)$, the \textit{submersion set} $\sub(T)$ is the set defined as
\[\sub(T)=\{\sigma \subseteq \gamma_d: \sigma\leq_{d+1} T,\, \dim(\sigma)=\lceil d/2 \rceil \}.\]
Submersion sets were used to provide the following encoding result which in particular implies that $\HST(n,d)$ is a lattice for $d=2,3$, see \cite[Theorems 3.6, 4.9]{edelman_cyclic_triangulation_envelope}. This encoding is equivalent to the one provided in \cite{HUANG19727}  for $d=2$.

\begin{thm} \label{lemma_lattice}
	For $d=2$ or $3$, the poset map $\Phi$ from $\HST(n,d)$ to the poset $\PP=\left(2^{\binom{[n]}{\lceil d/2 \rceil+1}}, \subseteq\right)$
	given by $\Phi:T \mapsto \sub(T)$ is an injection, and the subposet of $\PP$ induced by the image of $\Phi$ is a lattice where the meet $\sub(T_1) \wedge \sub(T_2)$ is given by the intersection $\sub(T_1)\cap \sub(T_2)$.	In particular, $\HST(n,d)$ is a lattice.
\end{thm}
The join in $\HST(n,d)$ is determined by the meet 
since $\HST(n,d)$ has a maximum.
In order to prove that $\Phi$ 
is injective, one of the key ideas in \cite{edelman_cyclic_triangulation_envelope} is to extract ``maximal'' elements from a submersion set, or more precisely, \textit{diagonal-closed} or \textit{triangle-closed} families, and (re)construct the triangulation from them. 
In a similar manner, we reformulate and prove the following for our purpose without considering the closed families after having a simple observation.

\begin{lemma} \label{lemma_IMT}
For $T\in S(n,d)$ and a simplex $\tau \subseteq \gamma_d$ with $\tau \leq_{d+1} T$, 
$\tau$ is a face of $T$ if and only if 
there are no $\sigma \in \sub(T)$ such that $\tau<_{d+1}\sigma$.
\end{lemma}

\begin{proof}
The only if implication is obvious since whenever $\sigma \in \sub(T)$ and $\tau \in T$ overlap in $\RR^d$, we have $\sigma <_{d+1} \tau$ by the definition of $\sub(T)$ and Lemma \ref{lemma_triangulation_simplex_simple}, which forbids  $\tau<_{d+1}\sigma$ by Corollary \ref{cor_cases}.
For the opposite direction, suppose the condition at the right-hand side. For any $\sigma' \in T$ of size $\lceil d/2\rceil+1$, since $\sigma'$ belongs to $\sub(T)$, we cannot have $\tau <_{d+1} \sigma'$. This implies that $\sigma \in T$ and $\tau$ never overlap in $\RR^d$ otherwise we have $\tau <_{d+1} \sigma'$ for some $\sigma' \subseteq \sigma$ of size $\lceil d/2\rceil+1$ by the assumption $\tau \leq_{d+1} T$, Lemma \ref{lemma_triangulation_simplex_simple}, and Proposition \ref{prop_height_interlacing}. 
Therefore, if $\conv(\tau)$ and $\conv(\sigma)$ intersect in their relative interiors for some $\sigma \in T$, 
they should have $\tau=\tau\cap \sigma=\sigma$. This implies $\tau \in T$.
\end{proof}

\section{Extending simplices into a triangulation} \label{sec_lattice}
In this section we prove Theorem~\ref{thm_main}(i). There are two ingredients in the proof: the fact that $\HST(n,d)$ is a
lattice for $d=2,3$, Lemma~\ref{lemma_lattice}; and the following theorems showing
the existence of triangulations in $\HST(n,2)$ and $\HST(n,3)$ with some special properties. 

\begin{thm} \label{thm_level_d=2}
Let $\sigma \subseteq \gamma_2$ be an edge or a triangle whose vertices are contained in $[n]$. Then there is a triangulation $T(\sigma) \in S(n,2)$ such that $\sigma \in T(\sigma)$ and for any triangle or edge $\tau$ different from $\sigma$ 
which does not overlap $\sigma$ in $\RR^2$ or does satisfy $\tau <_3 \sigma$, we have $\tau \leq_3 T(\sigma)$.
\end{thm}

\begin{thm} \label{thm_level_d=3}
Let $\sigma, \tau_1, \tau_2, \dots, \tau_m \subseteq \gamma_3$ be distinct triangles whose liftings are pairwise non-overlapping in $\RR^4$ and whose vertices are contained in $[n]$. We further assume that $\min(\sigma)\leq \min(\tau_i)$ for every $i\in [m]$, and $\min(\sigma)=\min(\tau_i)$ implies $\max(\sigma) \geq \max(\tau_i)$ for every $i \in [m]$. Then there is a triangulation $T \in S(n,3)$ such that $\sigma \in T$ and  $\tau_i \leq_4 T$ for every $i \in [m]$. 	
\end{thm}

We first present the proof of Theorem~\ref{thm_main}(i) assuming these facts in Subsection~\ref{subsec_proof_main}. The proofs of Theorems \ref{thm_level_d=2} and \ref{thm_level_d=3} are given in Subsections \ref{subsec_2-dim} and \ref{subsec_d=4}, respectively. The proof of Theorem \ref{thm_level_d=3} is a bit involved and depends on Lemma \ref{lemma_LMR_LMRT} which is a purely combinatorial claim. Lemma \ref{lemma_LMR_LMRT} and its proof are given in Subsection \ref{subsec_combinatorial_lemma}. Discussions delivered in Subsection \ref{subsec_combinatorial_lemma} are independent of other parts of this paper.

\subsection{Proof of Theorem \ref{thm_main}(i) modulo the existence of suitable triangulations}
\label{subsec_proof_main}
The following proposition enables us to reduce $\FF$ from Theorem~\ref{thm_main}(i) to consist only of simplices of dimension two when $D=4$, and of dimensions two or one when $D=3$.

\begin{prop} \label{prop_d/2-skeleton}
If a triangulation $T$ of a $D$-polytope contains the $\lceil D/2 \rceil$-dimensional skeleton of a $k$-simplex $\sigma$ for some $k$ with $\lceil D/2 \rceil \leq k\leq D$, then $T$ contains $\sigma$ as well.
\end{prop} 
\begin{proof}
	Suppose not. Then there is a facet $\tau\in T$ which overlaps $\sigma$ in $\RR^D$, hence $\sigma$ and $\tau$ are $(D+2)$-interlacing by Proposition \ref{prop_overlap}. Therefore we can find $\sigma'\subseteq \sigma$ of size $\lfloor (D+2)/2 \rfloor$ or $\lceil (D+2)/2 \rceil$ 
    which overlaps $\tau$ in $\RR^D$. But by assumption we have $\sigma'\in T$ as well as $\tau\in T$, which leads to a contradiction. 
\end{proof}

We temporarily assume the validity of Theorems \ref{thm_level_d=2} and \ref{thm_level_d=3} (proved in Subsections~\ref{subsec_2-dim} and \ref{subsec_d=4} respectively)
and use them to prove Theorem \ref{thm_main}(i).

\begin{proof}[Proof of Theorem \ref{thm_main}(i)]
When $D=3$, we may assume $\FF$ consists only of triangles and edges: indeed, for a 3-dimensional simplex $\sigma\in \FF$, we can remove $\sigma$ from $\FF$ and add the four triangles $t_1, t_2, t_3$ and $t_4$ from the boundary of $\sigma$. If a desired triangulation $T$ contains $t_1, \dots, t_4$ then $T$ must contain $\sigma$ as well by Proposition \ref{prop_d/2-skeleton}. 
When $D=4$, we may assume $\FF$ consists only of triangles:
a 4-dimensional cyclic polytope contain all possible edges on its vertex set at its boundary, thus we can remove them from $\FF$. Hence, by Proposition \ref{prop_d/2-skeleton} again, we only need to consider triangles, as for every $k$-simplex $\sigma\in \FF$ with $k \geq 3$, we replace $\sigma$ with its 2-faces. Obviously, for both $D=3,4$, after this replacement all simplices in the resulted collection of faces are again pairwise non-overlapping in $\RR^D$.

Let $d=D-1$ and $m=|\FF|$. We project the simplices in $\FF$ into $\RR^d$ by ignoring the last coordinate. 
Abusing notation, we also identify such projections with subsets of $[n]$.

We order the projections as $\sigma_1, \sigma_2, \dots, \sigma_m$ so that the ordering satisfies the following conditions (which depend on dimension $D$): 
When $D=3$, we require that for every $1\le i<j\le m$ either $\sigma_i <_D \sigma_j$ or $\sigma_i$ and $\sigma_j$ do not overlap in $\RR^d$;  such an ordering can be given by a linear extension of the partial order $\preceq_D$ guaranteed in Lemma \ref{lemma_linear_extension}.
Certainly in that linear extension, we cannot have $\sigma_j <_D \sigma_i$ for $1 \le i<j \le m$.
Also, the liftings $\hat\sigma_i$ and $\hat\sigma_j$ do not overlap in $\RR^D$ by the assumption. So by Corollary \ref{cor_cases}, either $\sigma_i <_D \sigma_j$ or $\sigma_i$ and $\sigma_j$ do not overlap in $\RR^d$. When $D=4$, we require that whenever $i<j$ we have $\min(\sigma_j) \leq \min(\sigma_i)$, and if $\min(\sigma_i)= \min(\sigma_j)$ then $\max(\sigma_i)\leq \max(\sigma_j)$. This ordering is clearly possible. Therefore by Theorems \ref{thm_level_d=2} and \ref{thm_level_d=3},  we obtain $T_i \in S(n, d)$ for every $i \in [m]$ such that $\sigma_i \in T_i$ and $\sigma_j \leq_D T_i$ for $j<i$.

Now we use the lattice property of $\HST(n,d)$ guaranteed by Theorem \ref{lemma_lattice}. For every $k\in [m]$, let $S_k=T_k \wedge T_{k+1} \wedge \cdots \wedge T_m$ where $\wedge$ denotes the meet operation. 
Clearly these $S_k$ form a chain in $\HST(n,d)$; considering a maximal chain in $\HST(n,d)$ containing all $S_k$ gives a triangulation $\widetilde T \in S(n,D)$ which contains the liftings of all $S_k$ by Theorem \ref{lemma_maximal_chain}.

It remains to show that $S_k$ contains $\sigma_k$ for every $k\in [m]$. 
It is enough to show that for every 
$\lceil d/2 \rceil$-dimensional face $\tau$ of $\sigma_k$, we have $\tau\in \sub(S_k)$ and $\tau$ satisfies the right-hand side condition of Lemma \ref{lemma_IMT} with $T$ replaced by $S_k$; once it is shown, then by Lemma \ref{lemma_IMT} every $\lceil d/2 \rceil$-dimensional face of $\sigma_k$ belongs to $S_k$.
Depending on the dimension of $\sigma_k$ there are two cases: (a) $\sigma_k$ itself is of dimension $\lceil d/2 \rceil$, namely an edge when $d=2$ or a triangle when $d=3$, where the proof is already complete; or (b) $\sigma_k$ is a triangle and $d=2$ where we use Proposition \ref{prop_d/2-skeleton} to derive the same conclusion.
 
Now,
\begin{align}
\sub(S_k)=\sub(\bigwedge_{i=k}^m T_i)=\bigcap_{i=k}^m\sub(T_i) \label{eq_lattice}
\end{align}
where the second equality is by Theorem \ref{lemma_lattice}. Since $\sigma_k\in T_k$ and $\sigma_k \leq_D T_i$ for every $i \in [k, m]$, we have $\tau \in \sub(T_i)$ for every $i \in [k,m]$, which together with (\ref{eq_lattice}) implies that $\tau \in \sub(S_k)$. Now assume for a contradiction (by Lemma \ref{lemma_IMT}), that there is another $\sigma' \in \sub(S_k)$ 
such that $\tau<_{d+1}\sigma'$.
In particular, $\sigma'$ and $\tau$ overlap in $\RR^d$.
Now by (\ref{eq_lattice}), we have $\sigma' \in \sub(T_k)$.  By definition of $\sub(T_k)$ and Lemma \ref{lemma_triangulation_simplex_simple} we have $\sigma'<_D \tau$, 
so we have a contradiction since $\sigma'<_D \tau$ and $\tau<_D \sigma'$ are exclusive events by Corollary \ref{cor_cases}. 
\end{proof} 
 
	\subsection{Proof of Theorem \ref{thm_level_d=2}} \label{subsec_2-dim}

	In this section, we prove Theorem \ref{thm_level_d=2}. While it is easy to construct a desired triangulation when $d=2$, it also serves as an important step towards the much subtler Theorem \ref{thm_level_d=3}.

    For convenience we spell out Lemma~\ref{lemma_triangulation_simplex} for the case $d=2$; when additionally $\sigma$ is an edge it gives \cite[Proposition 3.2]{edelman_cyclic_triangulation_envelope}).
    
    \begin{lemma} \label{lemma_triangulation_simplex_d=2}
		For every triangulation $T \in S(n,2)$ and every simplex $\sigma \subseteq \gamma_2$ there holds: $\sigma \leq_3 T$ ($\sigma \geq_3 T$, resp.) if and only if whenever an edge $e \subseteq \sigma$ crosses an edge $e' \in T$ we have $e<_3 e'$ ($e>_3 e'$, resp.).
\end{lemma}
	
	\begin{proof}[Proof of Theorem \ref{thm_level_d=2}]
				\begin{figure}[ht]
			\centering
			\includegraphics[totalheight=2.8cm]{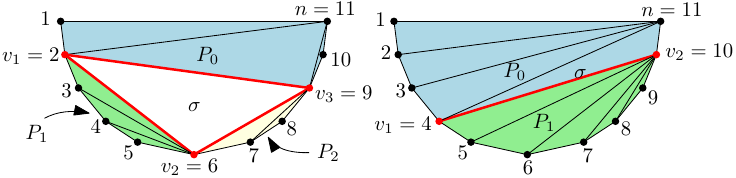}
			\caption{Extension of $\sigma$ into a triangulation $T(\sigma) \in S(n,2)$.
			}
			\label{fig:max_triangulation}
		\end{figure}
			Let $\sigma \subseteq \gamma_2$ be an edge or a triangle given in Theorem \ref{thm_level_d=2}. We denote $\sigma=\{v_1<v_2\}$ or $\sigma=\{v_1<v_2<v_3\}$ depending on whether $\sigma$ is an edge or a triangle.
		$\sigma$ dissects the region of $C(n,2)$ (outside $\sigma$ when $\sigma$ is a triangle) into at most 3 polygons $P_i$ which we need to further dissect into triangles, see Figure \ref{fig:max_triangulation}. For each polygon $P_i$, we add an edge between its maximum vertex $m_i$ (with respect to $<$) and its other vertices which were not previously adjacent to $m_i$. 
        For example, at the left of Figure \ref{fig:max_triangulation} the maximum vertices of $P_0$, $P_1$ and $P_2$ are $n=11$, $v_2=6$ and $v_3=9$, resp., while at the right the maximum vertices of $P_0$ and $P_1$ are $n=11$ and $v_2=10$, resp. This gives a triangulation $T=T(\sigma)$ of $C(n,2)$.
		
		It remains to show that if either $\tau$ does not overlap $\sigma$ in $\RR^2$ or $\tau<_3 \sigma$, then we have $\tau \leq_3 T$. 
        By using Lemma \ref{lemma_triangulation_simplex_d=2}, it is enough to show that for every edge $e\subseteq \tau$ must hold $e<_3 e'$ whenever $e$ crosses an edge $e' \in T$. When $e'\subseteq \sigma$ or $e'$ is an interior edge inside $P_0$, it is obvious. So we may assume $e'$ is an interior edge inside $P_1$ (or $P_2$ when $\sigma$ is a triangle, but this case can be argued similarly). Suppose otherwise that $e'<_3 e$ (using Corollary \ref{cor_cases}). Then we have $v_1\leq \min(e')<\min(e)<v_2=\max(e')<\max(e)$ by Proposition \ref{prop_height_interlacing}, which implies $v_1v_2 <_3 e$, contradicting our assumption on $\sigma$ and $\tau$.\end{proof}

This completes the proof of Theorem \ref{thm_main}(i) for $D=3$. 
The following corollary will be used in the proof of Theorem \ref{thm_level_d=3}.

\begin{cor} \label{cor_separating_level}
Let $E_1$, $E_2$ and $E_3$ be 3 disjoint sets of edges, not necessarily non-empty, on $\gamma_2$ with vertices contained in $[n]$. Suppose that the edges of $E_1 \cup E_2 \cup E_3$ are linearly ordered as $e_1, \dots, e_m$ such that
\begin{itemize}
    \item edges of $E_i$ come earlier than edges of $E_j$ whenever $i<j$,

    \item for $i,j \in [m]$ with $i<j$, we have either $e_i<_3 e_j$ or $e_i$ and $e_j$ do not cross, and

    \item the edges of $E_2$ are mutually non-crossing.
\end{itemize}
Then there is a triangulation $T\in S(n,2)$ such that we have
$e \leq_3 T$ for $e \in E_1$, $e \geq_3 T$ for $e \in E_3$, and $e \in T$ for $e \in E_2$. 
\end{cor}
The conditions in the corollary imply the following 
(forbidden) interlacing patterns: (i) edges in $E_2$ are never 4-interlacing, and (ii) for $1\le j<j'\leq 3$, if $e=\{x<y\}\in E_j$ and $e'=\{x'<y'\}\in E_{j'}$ are 4-interlacing  then $x<x'<y<y'$. Conversely, if (i) and (ii) hold then the existence of a linear order on $E_1\cup E_2\cup E_3$ as specified in the corollary is guaranteed: indeed, in each $E_i$ consider a linear order $\preceq_3^i$ that extends $\preceq_3$ on $E_i$, and concatenate these orders so that $\preceq_3^1$ comes first, $\preceq_3^2$ second and $\preceq_3^3$ third, to obtain the desired linear order on 
$E_1\cup E_2\cup E_3$. This is used in the proof of Proposition \ref{prop_only_LR} below.
\begin{proof}
Assume that all $E_i$ are non-empty; the other cases can be argued similarly. By Theorem \ref{thm_level_d=2}, we obtain a triangulation $T_i=T(e_i)$ 
of $C(n,2)$
for every $i \in [m]$ so that $e_i\in T_i$ and $e_j \leq_3 T_i$ for $1 \leq j<i\leq m$. Let $e_r$ be the first edge of $E_2$ in the linear ordering, and let $T=\bigwedge_{i=r}^m T_i$. Using a similar argument as in the proof of Theorem \ref{thm_main}(i) we obtain $e_r \in T$. Not only that, for every $e_i \in E_2$ we have $e_i \in T$: since $e_i$ does not cross any other $e_j \in E_2$ we have $e_i \leq_3 T_j=T(e_j)$, or equivalently $e_i \in \sub(T_j)$ by Theorem \ref{thm_level_d=2}. Thus as in the proof of Theorem \ref{thm_main}(i) we conclude that $e_i \in \sub(T)$ by (\ref{eq_lattice}) (using Lemma \ref{lemma_lattice}) but $e_i$ does not have any other $e' \in \sub(T)$ with $e_i<_3 e'$ (since $e_i \in T_i$ and $e' \in \sub(T_i)$), which implies $e_i \in T$ (using Lemma \ref{lemma_IMT}). 

By (\ref{eq_lattice})  we similarly obtain that every $e_i$ with $i<r$, or equivalently every $e_i \in E_1$, is in $\sub(T)$, or equivalently $e_i \leq_3 T$. For an index $i$ with 
$e_i \in E_3$, there holds $T\leq_3 T_i$, for $T$ being the meet of $T_i$ and other triangulations.
Since $e_i \in T_i$, we have $T \leq_3 e_i$. 
\end{proof}

Theorem~\ref{thm_level_d=2} and its proof implies that the join of all the triangulations in the lattice $\HST(n,2)$ which contain $\sigma$ again contains $\sigma$ -- indeed, this join equals $T(\sigma)$ in the proof of Theorem~\ref{thm_level_d=2}. However, this is not the case for $\HST(n,3)$, which leads us to take a more dynamic approach towards Theorem \ref{thm_level_d=3}.

\begin{eg}
In $\HST(7,3)$, the following two triangulations $T$ and $T'$ are maximal among all triangulations in $S(7,3)$ which contain the simplex $\sigma=\{2,5,6,7\}$, see \cite[Figure 4(b)]{edelman_cyclic_triangulation_envelope}:
\begin{align*}
    T&=\{\{1,2,3,5\}, \{1,2,5,6\}, \{1,2,6,7\}, \{1,3,4,5\}, \{2,3,5,7\}, \{2,5,6,7\}, \{3,4,5,7\}\}\\
    T'&=\{\{1,2,3,4\}, \{1,2,4,5\}, \{1,2,5,6\}, \{1,2,6,7\}, \{2,3,4,7\}, \{2,4,5,7\}, \{2,5,6,7\}\}
\end{align*}
\end{eg}

\subsection{Proof of Theorem \ref{thm_level_d=3}} \label{subsec_d=4}
Now we start the proof of Theorem \ref{thm_level_d=3}. In the proof we use combinatorial Lemma \ref{lemma_LMR_LMRT}, but its statement and proof are a bit involved. Thus we postpone them to the next Subsection \ref{subsec_combinatorial_lemma}, and rather give a clear outline of the proof of the theorem first by temporarily assuming Lemma \ref{lemma_LMR_LMRT}.

Let $\sigma, \tau_1, \dots, \tau_m \subseteq  \gamma_3$ be the given triangles from Theorem~\ref{thm_level_d=3}. Recall that they are all subsets of $[n]$. We may assume $n \geq 4$. Let $\sigma=\{v_1<v_2<v_3\}$.
Here is a restatement of Lemma \ref{lemma_triangulation_simplex} for $d=3$, which appeared already in \cite[Proposition 4.1]{edelman_cyclic_triangulation_envelope}.

\begin{lemma} \label{lemma_triangulation_simplex_d=3}
	For a triangulation $T \in S(n,3)$ and a triangle $\tau  \subseteq \gamma_3$, there holds: $\tau \leq_4 T$ if and only if there is no edge $e' \in T$ such that $e'<_4 \tau$, equivalently, such that $e'$ and $\tau$ are $5$-interlacing.
\end{lemma}
By Lemma \ref{lemma_triangulation_simplex_d=3}, 
 it is enough to find a triangulation $T$ which contains $\sigma$ and has no edges which are $5$-interlacing with any $\tau_i$.

\smallskip

We introduce some conventions which are useful in further discussion. For a vertex subset $V \subseteq \gamma_d$, $|V|\geq d+1$, let $S(V,d)$ be the set of triangulations of the $d$-dimensional cyclic polytope $C(V,d):=\conv(V)$ (without adding new vertices) and let $\HST(V,d)$ be the higher Stasheff-Tamari poset on $S(V,d)$. For simplicity of notations, if there is no confusion, we often identify $V \subseteq \gamma_d$ with its projection on $\gamma_{d'}$ or lifting on $\gamma_{d''}$ for $d'<d<d''$. 
So we may use notations like $S(V,d+1)$ or $S(V,d-1)$ even when $V \subseteq \gamma_d$. 

For a simplicial polytope $P$ and a point $q$ outside $P$ in $\RR^d$ 
such that $\{q\}\cup V(P)$ is in general position,
we say that a face $\tau$ (as a vertex subset) of $P$  \textit{is visible from $q$} if $\conv(\tau \cup \{q\})$ does not intersect the interior of $P$. For a tringulation of $T$ of $P$ (without adding new vertices), the \textit{cone of $T$ over $q$}, denoted by $\cone(T,q)$, is the triangulation of $\conv(P \cup \{q\})$ given by the set of faces 
\[T\cup \{\tau\cup \{q\}: \textrm{$\tau$ is a face of $P$ visible from $q$}    \}.\]

We make some simplifications of our setting. Using the coning operation, we may assume that $v_1=1$. Indeed when $v_1\geq 2$, if $T^1$ is a triangulation on $[v_1,n]$ with the desired properties, then we recursively define $T^i=\cone(T^{i-1}, q_i)$ where $q_i=v_1-i+1$  for $i \in [2, v_i]$. Whenever we cone $T^{i-1}$ over $q_i$, the newly added edges to $T^i$ are exactly the 2-subsets of the form $q_iw$ with $q_i<w\leq n$, since they appear as an edge of the cyclic polytope $P^i=C([q_i, n], 3)$ by Proposition \ref{prop_gale}.
These edges cannot be 5-interlacing with given triangles $\tau_i \subseteq [v_1, n]$ for every $i\in [m]$ (if $m>0$).
This argument also applies for the degenerate case when $|[v_1,n]|=3$, where $m=0$ by the assumptions.
Thus, w.l.o.g. $v_1=1$ and $|[v_1,n]| \geq 4$.

We define four intervals as follows.
\begin{align*}
J_L=&\{i\in [n] : i<v_2 \}\ne \emptyset, &J_R=\{i\in [n] : i>v_3\},\\
I_M=&\{i\in [n] : v_2<i<v_3\}, \textrm{ and } &J_M=I_M\cup \{v_2, v_3\}\ne\emptyset.
\end{align*}
Note that $I_M$ or $J_R$ might be empty. See Figure \ref{fig:4-dim}. 

				\begin{figure}[ht]
	\centering
	\includegraphics[totalheight=2.8cm]{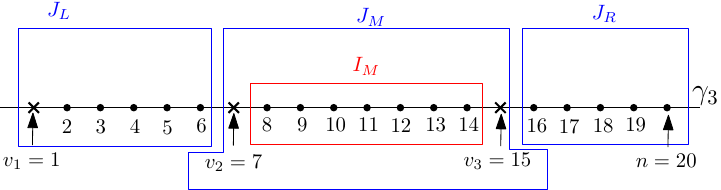}
	\caption{Illustrations of $\sigma=\{v_1, v_2, v_3\}$ and the intervals $J_L$, $J_R$ and $I_M\subset J_M$. 
	}
	\label{fig:4-dim}
\end{figure}

\begin{proof}[Proof of Theorem \ref{thm_level_d=3}] The proof proceeds in several steps. 

\smallskip 
	
\textbf{1. Initial triangulation.} 
For a vertex set $W \subseteq \gamma_3$ with $|W|\geq 4$, let us denote the maximum triangulation of $\HST(W,3)$ by $T_M^W$. Then $T_M^W$ is the projection of the upper envelope of the cyclic polytope on $\gamma_4$ on the vertices lifted from $W \subseteq \gamma_3$. 

Let $V_0=[n]\setminus I_M$. When $V_0=\sigma=\{1=v_1< v_2< v_3\}$, we have $v_3=n$ and $v_2=v_1+1$, so Proposition \ref{prop_gale} implies $\sigma$ is a face of the cyclic polytope $C(n,3)$. Hence $\sigma$ is a face of any triangulation in $S(n,3)$, and in particular a face of $T_M^{[n]}$. Since the lifting of $T_M^{[n]}$ is the upper envelope of $C(n,4)$ we have $\tau_i \leq_4 T_M^{[n]}$ for every $i\in [m]$. Therefore $T_M^{[n]}$ is a desired triangulation for this case.

Hence we may assume that $|V_0| \geq 4$. 
Let $T_0=T_M^{V_0}$. By Proposition \ref{prop_gale}, the facets of $T_0$ are given as
\[\{\{1,n\}*\{i,s(i)\}: i \in V_0\setminus \{1,n, \max(V_0\setminus \{1,n\})\}\},\]
where $s(i)$ is the immediate successor of $i$ in $V_0$ with respect to $<$.
Since $v_1=1$ and $v_3=s(v_2)$, the above characterization implies $\sigma \in T_0$. From the above description of the facets of $T_0$ one easily checks that
the set of edges of $T_0$ is 
\begin{equation}\label{eq_edges_primary_tri}
\{\{i,n\}, \{1,j\}, \{k,s(k)\}: i\in V_0\setminus \{n\}, j \in V_0\setminus \{1\}, k\in V_0\setminus \{1,n\}\} \}.
\end{equation}
None of the edges in (\ref{eq_edges_primary_tri})
is $5$-interlacing with any triangle $\tau_i$ when $I_M$ is empty, so we can take $T=T_0$ 
in the assertion of the theorem in this case. 

Thus, we may assume that $l:=|I_M|>0$. Still the edges other  than $\{v_2,v_3\}\subseteq J_M$ cannot 5-interlace any of the given triangles $\tau_i$.

\smallskip 

\textbf{2. New edges after inductive coning.}
We will later find a certain total order on $I_M$ and use it to find a desirable triangulation. Here we first see how this order is used to construct a triangulation; suppose that all the $l$ elements of $I_M$ are ordered as $q_1, q_2, \dots, q_l$. For $i\in [l]$, let $V_i=V_{i-1}\cup \{q_i\}$
and $T_i=\cone(T_{i-1}, q_i) \in S(V_i,3)$.

The edges newly added at the $i$th step to form $T_i$ are exactly $\{1, q_i\}$, $\{p(q_i), q_i\}$, $\{q_i, s(q_i)\}$ and $\{q_i, n\}$ where $p(q_i)$ and $s(q_i)$ are the immediate predecessor and successor of $q_i$ in $V_i$, respectively. These edges are exactly the edges of the cyclic polytope $P_i=\conv(V_i)$  which are not in $T_{i-1}$ (by Proposition \ref{prop_gale}); the other edges among $q_iw$ for $w\in V_{i-1}$ do 5-interlace, or overlap in $\RR^3$, some facet of $P_{i-1}$, so the characterization of the new edges holds.
Of course the edges $\{1, q_i\}$ and $\{q_i, n\}$ can never be $5$-interlacing with any triangles on $[n]$. 
Thus, also after all inductive conings by $q_1,\ldots,q_l$, the edges of the resulting triangulation which 
potentially 5-interlace some triangle $\tau_i$ are just all the 2-subsets of $J_M$ (including $v_2v_3$), so we only need to consider them.

In fact, it is easily seen that these new edges together with $v_2v_3$ form a 2-dimensional triangulation on $J_M$, that is, an element of $S(J_M,2)$. 
The other direction also holds easily as follows.

\begin{claim}\label{claim_cone_edge}
Let $l=|I_M|>0$. 
Then for every trinangulation $T \in S(J_M, 2)$, there exists an ordering $q_1, \dots, q_l$ on $I_M$ such that for the resulting triangulation $\widetilde T=T_l\in S(V,3)$ after inductive conings following this order starting from $T_0$, the edges of $\widetilde T$ contained in $J_M$ are exactly the edges of $T$.
\end{claim}
\begin{proof}
Both $T$ and $T_0$ contain the edge $v_2v_3$, and the outerplanar triangulation $T$ has a unique triangle $t_1$ containing the edge $v_2v_3$, call $q_1$ its third vertex.
Denote by $T^*$ the directed tree dual to $T$, rooted at the triangle $t_1$ incident with $v_2v_3$. The vertices of $T^*$ are the triangles of $T$. For every vertex $t$ of $T^*$, we have a unique directed path from $t_1$ to $t$ in $T^*$, $t_1 \rightarrow t_2 \rightarrow \cdots \rightarrow t_s=t$, and there is the unique vertex of $T$, denote it $p_t$, which never appeared in the previous triangles $t_1,\ldots, t_{s-1}$ but belongs to $t$. In this way the vertices of $I_M$ are in 1-1 correspondence with the triangles of $T$. We define a partial order $\preceq$ on $I_M$ such that $p_{t} \preceq p_{t'}$ when there is a directed path from $t$ to $t'$ in $T^*$. Fix a linear extension of $\preceq$, $q_1, \dots, q_l$ so that for $i<j$, either $q_i \preceq q_j$ or $q_i$ are incomparable with respect to $\preceq$. This order respects the order of triangles when constructing $T$ by stacking. 
Thus, this order on $I_M$ gives a triangulation $\widetilde T$ where its new edges that are contained in $J_M$ are exactly the edges of $T$ except for $v_2v_3$.
\end{proof}

\smallskip

\textbf{3. Special triangles and edges.} 
We are left to find a triangulation $T \in S(J_M,2)$ where none of its edges is 5-interlacing with some $\tau_i$ for $i\in [m]$. Once this is achieved, we construct $\widetilde T \in S(n,3)$ from the initial triangulation $T_0$ by inductive coning following the order given by Claim \ref{claim_cone_edge} using $T$. Claim \ref{claim_cone_edge} implies that $\widetilde T$ has no edges $5$-interlacing with any triangle $\tau_i$. Therefore we have $\tau_i \leq_{4} \widetilde T$ for all $i\in [m]$ by Lemma \ref{lemma_triangulation_simplex_d=3}. Note also that $\sigma\in \widetilde{T}$. So $\widetilde{T}$ is as desired in the theorem and the proof concludes.

Let us consider what happens when an edge $e'=\{x_1<x_2\}$ in a triangulation $T' \in S(J_M,2)$ and a triangle $\tau=\tau_i=\{y_1<y_2<y_3\}\subseteq [n]$ (for some $i\in [m]$) are $5$-interlacing, namely $y_1<x_1<y_2<x_2<y_3$. In particular, this implies that $y_2 \in I_M$, hence $|\tau\cap J_M| \geq 1$. 
We argue by the size of $\tau\cap J_M$.

Suppose $|\tau\cap J_M| = 1$. Then $y_1 \in J_L$ and $y_3 \in J_R$. We have $y_1=1$, which implies $\min(\tau)=y_1=1=v_1=\min(\sigma)$, as otherwise $\tau$ and $\sigma$ would be $6$-interlacing which is forbidden by the non-overlapping assumption and Proposition \ref{prop_overlap}. But then $y_3 \in J_R$ yields a contradiction by the assumption $\max(\tau)=y_3 \leq \max(\sigma)=v_3$. 
Hence $|\tau \cap J_M|=2$ or $3$, and $\tau$ must be of exactly one of the following three types:

\begin{itemize}
	\item[(L)] $|\tau \cap J_L|=1$.
	\item[(R)] $|\tau \cap J_R|=1$.
	\item[(M)] $\tau \subseteq J_M$.
\end{itemize}
Let $L$ (resp. $R$) be the set of all edges of the form $\tau \cap J_M$ for $\tau \in \{\tau_1, \dots, \tau_m\}$ of type (L) (resp. (R)). Let $M$ be the set of all triangles $\tau$ of type (M) among $\tau_1, \dots, \tau_m$.

\begin{claim} \label{claim_LMR}
The sets $L$, $R$ and $M$ satisfy Conditions (LR), (LMR) and (MM) from Subsection \ref{subsec_combinatorial_lemma} with the consistent notations.
\end{claim} 
\begin{proof}
This is clear as no two triangles among 
$\tau_1, \dots, \tau_m$ are 6-interlacing.
\end{proof}
By Claim \ref{claim_LMR}, we can apply Lemma \ref{lemma_LMR_LMRT}, given in the next Subsection \ref{subsec_combinatorial_lemma}, by setting $V=J_M$ 
to obtain a triangulation $T \in S(J_M,2)$ such that each edge of $T$ satisfies Conditions (Le), (Re) and (Me) from Subsection \ref{subsec_combinatorial_lemma} with respect to the sets $L$, $M$ and $R$ defined above. 
These conditions exactly correspond to the non-5-interlacing assumption between the triangles $\tau_1, \dots, \tau_m$ and the edges of $T$. 
From such $T\in S(J_M,2)$ we construct $\Tilde{T}\in S(V,3)$ as desired, which completes the proof of the theorem.
\end{proof}

\subsection{Finding a 2-dimensional triangulation avoiding certain interlacing patterns} \label{subsec_combinatorial_lemma}
This subsection is devoted to the proof of Lemma \ref{lemma_LMR_LMRT} below, a crucial lemma in finishing the proof of Theorem \ref{thm_level_d=3}. It is a purely combinatorial claim.
Its proof is a bit involved, 
however, the proof itself is independent from the rest of this paper.

\smallskip

Let $L$ and $R$ be two sets of edges on the 2-dimensional moment curve $\gamma_2$, and $M$ be a set of triangles on $\gamma_2$. Again using conventions from Section \ref{sec_prelim}, following the natural linear order from $\gamma_2$ we identify an edge or a triangle from $L\cup M \cup R$ as a subset of $[n]$ for some  positive integer $n$. Note that the sets $L$ and $R$ might intersect, and any set among $L$, $M$ and $R$ might be empty. 

Regarding these $L$, $M$ and $R$, we consider the following forbidden interlacing conditions:

\begin{itemize}
\item[(LR)] No edges $\{l_1 < l_2\}$ of $L$ and $\{r_1<r_2\}$ of $R$ are 4-interlacing starting with $r_1$, i.e., $r_1<l_1<r_2<l_2$ never holds. 

\item[(LMR)] For every edge $\{v_1<v_2\}$ of $L \cup R$ and every triangle $\{w_1<w_2<w_3\}$ of $M$, they are not 5-interlacing, i.e., $w_1<v_1<w_2<v_2<w_3$ never holds. 

\item[(MM)]  For every pair of triangles $\{v_1< v_2< v_3\}$ and $\{w_1<w_2<w_3\}$ of $M$, they are not 6-interlacing, i.e., neither $v_1<w_1< v_2<w_2< v_3<w_3$ nor $w_1<v_1< w_2<v_2< w_3<v_3$ holds. 
\end{itemize}

For an edge $e=\{v_1<v_2\} \subseteq \gamma_2$ (as a subset of $[n]$), we define the following 
forbidden interlacing 
conditions w.r.t. the sets $L$, $M$ and $R$. Eventually, 
we will need to find a triangulation of the polygon $C(n,2)$ so that each edge of the triangulation satisfies the following three conditions; achieved in Lemma~\ref{lemma_LMR_LMRT}.

\begin{itemize}
\item[(Le)]  No edge $\{l_1<l_2\}$ of $L$ is 4-interlacing with $e$ starting with $v_1$, i.e., $v_1<l_1<v_2<l_2$ never holds. 

\item[(Me)]  No triangle $\{w_1<w_2<w_3\} \in M$ is 5-interlacing with $e$, i.e., $w_1<v_1<w_2<v_2<w_3$ never holds. 

\item[(Re)]  No edge $\{r_1<r_2\}$ of $R$ is 4-interlacing with $e$ starting with $r_1$, i.e., $r_1<v_1<r_2<v_2$ never holds. 
\end{itemize}

\begin{lemma} \label{lemma_LMR_LMRT}
Let $L$, $M$ and $R$ be finite families of edges or triangles on $\gamma_2$ given as above, satisfying (LR), (LMR), and (MM). Suppose there is a vertex subset $V \subseteq [n]$ on $\gamma_2$ with $|V|\geq 3$ such that each edge of the (convex) polygon $P=\conv(V)$ satisfies (Le), (Me) and (Re). Then we can find a triangulation $T$ of $P$ such that each edge of $T$ satisfies (Le), (Me) and (Re).
\end{lemma}

We may assume that $L \cup R$ does not contain any edges of $P$ since any diagonal of $P$ cannot cross such edges.

\begin{prop} \label{prop_only_LR}
Lemma \ref{lemma_LMR_LMRT} holds true when $M$ is empty.
\end{prop}
\begin{proof}
We let $E_1=L\setminus R$, $E_3=R\setminus L$ and $E_2$ be the union of the set of the edges of $P$ and of $L \cap R$. For each $i \in [3]$ we give a linear order on $E_i$ by using a linear extension of the partial order $\preceq_3$ guaranteed by Lemma \ref{lemma_linear_extension}. Then we concatenate them 
to a linear order on their union, namely that an edge of $E_i$ comes earlier than an edge of $E_j$ whenever $i<j$. It is easy to see that these three sets along with the linear order satisfy the conditions of Corollary \ref{cor_separating_level}, so we find a triangulation $T' \in S(n,2)$ with the properties at the conclusion. By Lemma \ref{lemma_triangulation_simplex_d=2} and Proposition \ref{prop_height_interlacing}, it is easy to see that 
Conditions (Le) and (Re) hold for each edge in $T'$. 
Further, all edges of $E_2$, in particular those of the polygon $P$, are in $T'$, and thus $T'$ contains a triangulation of $P$.
The restriction $T$ of $T'$ to $P$ is a desired triangulation.
\end{proof}

Before we prove Lemma \ref{lemma_LMR_LMRT} we introduce some conventions. 
The \textit{separating edge} of a vertex $w \in [n]\setminus V$, if there is such a $w$, is an edge $e$ of the polygon $P=\conv(V)$ which weakly separates $w$ from $V$ in $C(n,2)$. We say that two vertices $w_1$ and $w_2$ of $[n]$, not necessarily distinct, are \textit{at the same side} if we have $w_1=w_2$, or $w_1w_2$ is an edge of $P$, or at least one of them, say $w_1$, does not belong to $V$ and the separating edge of $w_1$ weakly separates $w_1$ and $w_2$ 
from the (other) vertices of $V$.

\begin{proof}[Proof of Lemma \ref{lemma_LMR_LMRT}]
We  call an element of $L$, $R$ and $M$ a \textit{left edge}, a \textit{right edge}, and a \textit{middle triangle}, respectively. We use induction on $|V|+|M|+|R|+|L|$. When $|V|=3$, there is nothing to prove. If there are no middle triangles, then we conclude using Proposition \ref{prop_only_LR}. So suppose that 
$|M|\ge 1$ and $|V| \geq 4$.

Our goal is to find a diagonal in the polygon $P$ which satisfies (Le), (Me) and (Re). Indeed, such a diagonal divides $P$ into smaller polygons, say $P_1$ and $P_2$. Since each of $P_1$ and $P_2$ have fewer vertices than $P$, using the induction hypothesis (with respect to the same sets $L$, $M$ and $R$ but with smaller vertex set) 
we have a desired triangulation of each $P_i$. The union of these two triangulations is a desirable triangulation of $P$, and we conclude. For finding such diagonal we will use also the induction on $|M|$, $|R|$ and $|L|$.  

W.l.o.g. we may further assume the followings, see Figure \ref{fig:middle_assumption_w2V}:
\begin{itemize}
    \item[(Ms)] For a middle triangle $t'=\{w_1'<w_2'<w_3'\}$, the middle vertex $w_2'$ is not at the same side of $w_1'$ nor of $w_3'$ (this in particular forbids the case when $w_2' \leq \min(V)$ or $\max(V)\leq w_2'$). 

    \item[(MV)] For a middle triangle $t'=\{w_1'<w_2'<w_3'\}$, the middle vertex $w_2'$ belongs to $V$.

    \item[(LV)] For a left edge $e_L=\{l_1<l_2\}$, the left vertex $l_1$ belongs to $V$ unless $l_1$ is at the same side of both $1$ and $n$.

    \item[(RV)] For a right edge $e_R=\{r_1<r_2\}$, the right vertex $r_2$ belongs to $V$ unless $r_2$ is at the same side of both $1$ and $n$.
\end{itemize}

\begin{figure}[ht]
			\centering
			\includegraphics[totalheight=3.5cm]{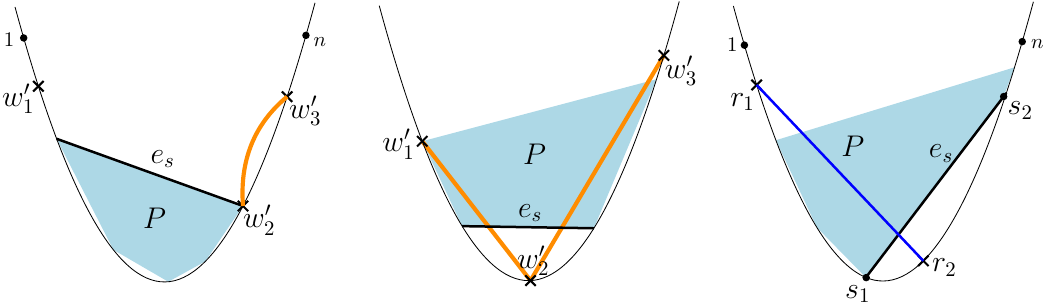}
			\caption{Left: case when $w_2'$ is at the same side of $w_1'$ or $w_3'$ against (Ms). In this case the triangle $t'=w_1'w_2'w_3'$ can be ignored. Middle:  case when $w_2' \notin V$ against (MV). In this case $e_s$ violates (Me). Right: case against (RV). In this case $e_s$ violates (Re) unless $s_1\leq r_1$, in which case we can ignore $e_R=r_1r_2$.}
			\label{fig:middle_assumption_w2V}
		\end{figure}
 Indeed, for (Ms), suppose w.l.o.g. that $w_2'$ is at the same side of $w_1'$. If $w_2'$ is also at the same side of both 1 and $n$ altogether, then we have either $w_2' \leq \min(V)$ or $\max(V) \leq w_2'$, which implies that $(w_1', w_2')\cap V$ or $(w_2', w_3')\cap V$ is empty, respectively (here $(\cdot,\cdot)$ denotes an open interval). This implies that $t'$ is not 5-interlacing 
 with any diagonals of $P$ that violates (Me). Therefore, we can ignore $t'$ from $M$, and this reduces $|M|$ so that we can use the induction hypothesis. For the other case we can easily see that $(w_1',w_2')\cap V$ is empty, and we conclude similarly.

For (MV), suppose otherwise that $w_2'\notin V$. By (Ms), 
$w_2'$ is not at the same side of $w_1'$, nor is it at the same side of $w_3'$.
In particular, $w_2'$ is not at the same side of both $1$ and $n$. This implies that the separating edge $e_s$ of $w_2'$ crosses both edges $\{w_1',w_2'\}$ and $\{w_2', w_3'\}$. This is a contradiction as $e_s$ satisfies (Me) w.r.t. $t'$.

For (RV), suppose otherwise that $r_2\notin V$ but $r_2$ is not at the same side of both $1$ and $n$. Let $e_s=\{s_1 < s_2\}$ be the separating edge of $r_2$. We cannot have $r_1 <s_1$ since then $e_s$ violates (Re) w.r.t. $e_R$. So we have $r_1 \geq s_1$, but in this case $V \cap (r_1, r_2)$ is empty. By the same argument as for (Ms), we can simply remove $e_R$, 
which reduces $|R|$, finishing the proof by the induction hypothesis. Similarly argue for (LV), to reduce  $|L|$ and use the induction hypothesis. 

\smallskip
\begin{figure}[ht]
			\centering
			\includegraphics[totalheight=3.5cm]{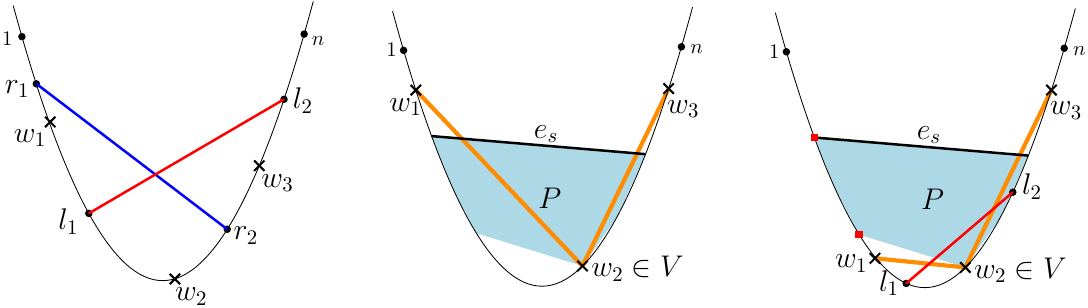}
			\caption{Left: 
            when neither (i) nor (ii) holds. Not to violate (LMR), we have $r_1\leq w_1$ and $w_3\leq l_2$. Then (LR) is violated. Middle: when $w_1<\min(V)$ and $\max(V)<w_3$. Then we violate (Ms) or (Me).
            Right: when $w_3>\max(V)$ and $\min(V)=m\leq w_1$, but (ii) is violated. Here depicted the case $l_2<w_3$ hence (LMR) is violated. Else $w_3 \leq l_2$ and then $e_s$ violates (Le).
           }
			\label{fig:comb_lemma}
\end{figure}

To each middle triangle $t' \in M$, define the interval $I_{t'}:=[\min(t'), \max(t')]$. Let $t=\{w_1<w_2<w_3\}$ be a middle triangle such that $I_t$ is maximal with respect to inclusion among all middle triangles. 
In this case one of the following two alternatives must hold:
\begin{itemize}
\item[(i)] No $\{r_1< r_2\}\in R$ is 4-interlacing with $w_2w_3$ starting with $r_1$, i.e., $r_1<w_2<r_2<w_3$ never holds.

\item[(ii)] No $\{l_1< l_2\} \in L$ is 4-interlacing with $w_1w_2$ starting with $w_1$, i.e., $w_1<l_1<w_2<l_2$ never holds.
\end{itemize}
Indeed, otherwise either (LMR) is violated or else (LR) is violated; see the left of Figure \ref{fig:comb_lemma}. Assume (i) holds; the case when (ii) holds it treated similarly. 

We will find a desirable diagonal of $P$ among the edges of the form $w_2q$ where $q$ satisfies $w_3\le q \in V$. We first show that we may assume there exists a vertex $q$ satisfying $w_3\le q \in V$. To see this, suppose otherwise. Then $\max(V)<w_3$ so $w_3$ is at the same side of both $1$ and $n$.
If additionally
$w_1<\min(V)$ (see the middle of Figure \ref{fig:comb_lemma}), 
then the separating edge of $n$ (which also separates  $1$, $w_1$ and $w_3$) violates (Me) w.r.t. the triangle $t\in M$, which leads to a contradiction (here we used (Ms) to ensure $\min(V)<w_2<\max(V)$). Thus, there is a vertex $q \in V$ with $q \leq w_1$. In this case (ii) also  holds, as otherwise 
for a left edge $e_L=\{l_1<l_2\}$ which violates (ii), the separating edge $e_s$ of $n$ would either violate (Le) with $e_L$ (when $w_3 \leq l_2$) or else $l_2<w_3$ hence $e_L$ and $t \in M$ violate (LMR), which again is a contradiction; see the right of Figure \ref{fig:comb_lemma}.
Therefore, by symmetry of cases (i) and (ii), we may assume that (i) holds and that there exists a vertex $q \in V$ with $q \geq w_3$.

Let $m$ and $m'$ be the maximum and minimum vertex among such $q$, respectively. 
Then $m'$ cannot be the immediate successor of $w_2$ in $V$ with respect to $<$ (see the right at Figure \ref{fig:comb_lemma} for a reflected illustration): Otherwise, since we have $\min(V)<w_2<\max(V)$ by (Ms), any vertex in $(w_2, m']$ is at the same side of $w_2$, and in particular so is $w_3$. This is forbidden by (Ms). 
Therefore, every edge $w_2q$ with $q \geq w_3$ and $q\in V$ is a diagonal of $P$, not an edge of $P$.

 We will choose a suitable $q$ following some decreasing sequence from $m$ so that $w_2q$ is a diagonal of $P$ satisfying (Le), (Re) and (Me). For this, first we observe the following.

\begin{figure}[ht]
			\centering
			\includegraphics[totalheight=2.8cm]{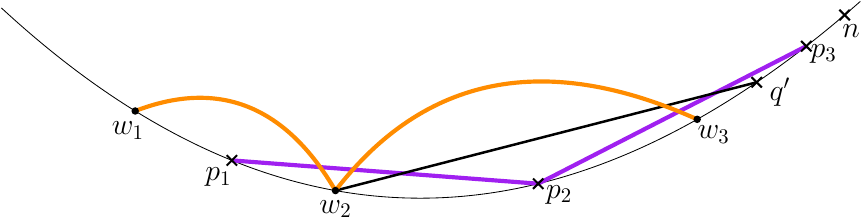}
			\caption{Illustration of Claim \ref{claim_middle_interlace_step}: why $p_2 < w_3$ contradicts (MM).}
			\label{fig:claim1}
		\end{figure}

\begin{claim} \label{claim_middle_interlace_step}
Let $\{p_1, p_2, p_3\} \in M$ be a middle triangle which is 5-interlacing with some $e'=w_2q'$ such that $q' \in [w_3,n]\cap V$.
Then $p_2 \geq w_3$. This further implies that $q' \in (w_3, n]$.
\end{claim}
\begin{proof}
As
$w_3 \leq q'<p_3$, by maximality of the interval $I_t=[w_1, w_3]$, we have $w_1<p_1$. Also, by the condition of the claim, we have $p_1<w_2<p_2$. Therefore, if additionally $p_2 < w_3$ it would contradics Condition (MM), hence $p_2 \geq w_3$. Since $w_3\leq p_2<q'$, we further have $q'>w_3$. 
\end{proof} 

\begin{figure}[ht]
			\centering
			\includegraphics[totalheight=5.6cm]{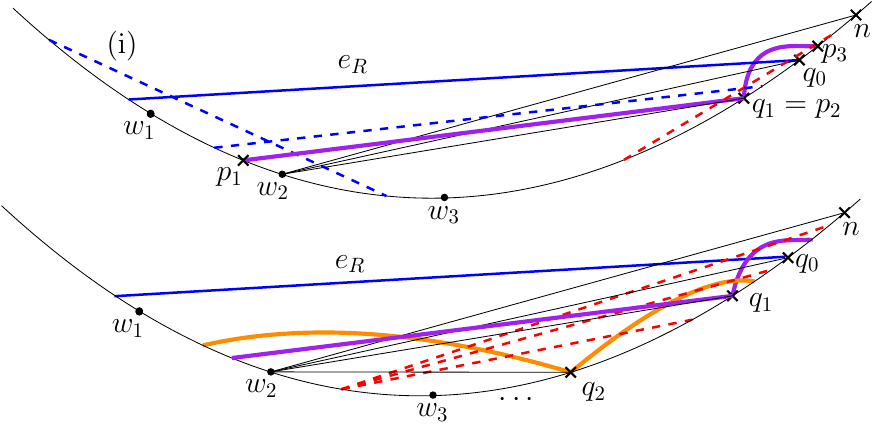}
			\caption{Illustration of the process of finding a suitable diagonal when $V=[n]$. Left edges are red,  right edges are blue, and when dashed these edges are forbidden. The thick purple and orange edges describe middle triangles found in the process. Thin black edges from $w_2$ are the edges $e_i$.\\
            Above: Right edges forbidden by (i), forbidden edges due to the right edge $e_R$ with the right vertex $q_0$, 
            and a middle triangle used in further steps to define $q_1$. In particular, after choosing the diagonal $w_2q_0$ right edges are not problematic later on. 
        Below: Inductive iteration of finding the edges $e_i$. The edge $e_2=w_2q_2$ does not violate (Le) because of (LMR) and (LR).}
			\label{fig:finding_e}
		\end{figure}

Our choice of $q$ depends on case analysis; see Figure \ref{fig:finding_e} for illustration.  
For a left edge $\{l_1<l_2\} \in L$ we cannot have $w_2<l_1<m<l_2$: When $m=n$ it is simply impossible. When $n \notin V$, the separating edge of $n$ (it contains $m$) violates (Le) w.r.t. $l_1l_2$. So if the diagonal $w_2m$ is not suitable it must violate (Me) or (Re).
Let $M_0$ and $R_0$ be the set of middle triangles and right edges w.r.t. which $w_2m$ violates (Me) or (Re)  respectively. If $M_0 \cup R_0 =\emptyset$, we choose the diagonal $w_2m$ and conclude the proof. So suppose  $M_0 \cup R_0$ is nonempty.  
If $M_0 \ne \emptyset$ then $n\ne m$, or equivalently $n \notin V$. 

For $\{r_1<r_2\} \in R_0$ the right vertex $r_2$ belongs to $V$ by (RV) and by $w_2<r_2<m$. Further, for $\{p_1<p_2<p_3\} \in M_0$ the middle vertex $p_2$ belongs to $V$ by (MV). Also, such right and middle vertices are at least $w_3$, by (i) and by Claim \ref{claim_middle_interlace_step} respectively.
Hence $q_0$, defined as the minimum among all the middle vertices of the middle triangles of $M_0$ and all the right vertices of the right edges of $R_0$, satisfies $w_3\le q_0 \in V$.

Let $e_0=w_2q_0$. Then $e_0$ satisfies (Le):  otherwise there is a left edge $e_L \in L$ such that $e_L$ either violates (LR) w.r.t. a right edge of $R_0$ 
(when $q_0$ as the right vertex of an edge in $R_0$),
or (LMR) with a middle triangle in $M_0$
(when $q_0$ is the middle vertex of a triangle in $M_0$ and the maximal vertex of $e_L$ is at most $m$), 
or else the separating edge of $n$ violates (Le) w.r.t. $e_L$ (as in this case $m<\max(e_L)$ and this separating edge contains $m$). 

Also by the minimality of $q_0$, $e_0$ does not violate (Re); in fact, for the same reason, right edges will not be an issue anymore when searching for $q\le q_0$.

Hence, if $e_0$ satisfies (Me), then we are done by choosing the diagonal $e_0$. Suppose then that is not the case, namely $e_0$ violates (Me).
We repeat this process of finding a new diagonal  inductively: Suppose we have found $e_{i-1}=w_2q_{i-1}$ for some $i \geq 1$ with $q_{i-1} \in [w_3, q_{i-2})\cap V$ (set $q_{-1}=m$ for the case $i=1$). 
Let $M_i$ be the set of middle triangles which are 5-interlacing with $e_{i-1}$. Choose the minimum middle vertex among all triangles from $M_i$ and denote it by $q_i$. Let $e_{i}=w_2q_{i}$. By our choice we have $q_i <q_{i-1}$. Claim \ref{claim_middle_interlace_step} implies that $q_i \geq w_3$. Also, $q_i \in V$ by (MV).

This inductive process cannot go forever so it should stop at some $e_k$. We claim that $e_k$ is our desired choice of diagonal of $P$: 
obviously $e_k$ is not 5-interlacing with any middle triangle, namely $e_k$ satisfies (Me). By the minimality of $q_0$ again, 
$e_k$ satisfies (Re). It is left to show that $e_k$ satisfies (Le).

Suppose otherwise for contradiction, namely, that $e_k$ violates (Le) w.r.t. some left edge $e_L= \{l_1<l_2\} \in L$, that is, we have $w_2<l_1<q_k<l_2$. For $i \in [k]$, let $J_i$ be the interval $(q_i,\max(t_i))$ where $t_i$ is the middle triangle used to determine $q_i$. By the interlacing assumptions between $e_{i-1}=w_2q_{i-1}$ and $t_i$ 
we have $q_{i-1} \in J_i$ for $i\in [k]$. 
Recall that for $e_0$ there were two cases:
(a) when $q_0$ is the middle vertex of a middle triangle $t_0$, and (b) when $q_0$ is the right vertex of a right edge $e_R$ in $R_0$.

In \textbf{Case (a)}:
let $J_0$ be similarly defined as $(q_0,\max(t_0))$. By the interlacing assumption again, we have $m \in J_0$. 
Then
$\bigcup_{i=0}^k J_i$ covers $(q_k, m]$. 
We distinguish two cases:
(a1) Case $l_2 \leq m$. Then there is some $s$ with $0 \leq s \leq k$ such that $l_2 \in J_s$. Since $\min(t_s)< w_2$, this implies that $e_L$ and $t_s$ violate (LMR) which is a contradiction. 
(a2) Case $l_2>m$. Then the separating edge of $n$ in $P$ violates (Le) w.r.t $e_L$, which is again a contradiction. 

In \textbf{Case (b)}:
since $\bigcup_{i=i}^k J_i$ covers $(q_k, q_0]$, if $l_2 \in (q_k, q_0]$ we can argue similarly to reach a contradiction. Else, $l_2>q_0$ hence $e_L$ violates (LR) with the right edge $e_R$ used to define $q_0$, a contradiction. 

The proof is complete. \end{proof}

\section{Proof of Theorem \ref{thm_main}(ii)} \label{sec_countereg}
In this section we prove Theorem \ref{thm_main}(ii), that is, for every $D\ge 5$ and $n \geq D+3$, we provide 
a collection of pairwise non-overlapping $D$-simplices on the moment curve $\gamma_D$ in $\RR^D$ on exactly $n$ vertices which is not extendable into a triangulation of the cyclic polytope without adding new vertices, or \textit{non-extendable} for short.

We first recall a small non-extendable example when $D=5$ and $n=D+3=8$ due to Rambau~\cite[Example 4.5]{rambau}, and include our own proof of it for completeness.
\begin{figure}[ht]
			\centering
			\includegraphics[totalheight=2.8cm]{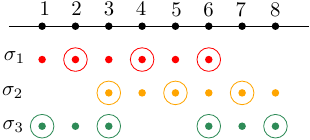}
			\caption{Illustration of Proposition \ref{prop_eg_D=5_n=8}. Circled elements in each $\sigma_i$ describe the only possible choice of elements which can form a $7$-interlacing pattern with other simplices.}
			\label{fig:d=5_n=8}
		\end{figure}
\begin{prop}(\cite[Example 4.5]{rambau}) \label{prop_eg_D=5_n=8}
   On the moment curve in $\RR^5$, one cannot add a new $5$-simplex in $[8]$ to the list
   \[\sigma_1=\{1, 2, 3, 4, 5, 6\}, \sigma_2=\{3, 4, 5, 6, 7, 8\}, \sigma_3=\{1, 2, 3, 6, 7, 8\}\]
   without overlapping one of the simplices in the list.
\end{prop}
\begin{proof}
    (See Figure \ref{fig:d=5_n=8}.) Suppose otherwise that there is a $5$-simplex $\tau$ on $[8]$ on $\gamma_5$  which is different from and does not overlap any $\sigma_i$ for $i\in [3]$ in $\RR^5$. As $\tau$ does not $7$-interlaces $\sigma_1$ nor $\sigma_2$, we have $|\tau \cap \{1,3,5\}| \leq 2$ and $|\tau \cap \{4,6,8\}| \leq 2$. This implies $2, 7 \in \tau$. Thus none of $4$ and $5$ are in $\tau$, otherwise $\tau$ would $7$-interlace with  $\sigma_3$, hence overlap $\sigma_3$. This forces $\tau=\sigma_3$, a contradiction.
\end{proof}

Next, we prove two propositions that extend the example in Proposition \ref{prop_eg_D=5_n=8} to prove Theorem~\ref{thm_main}(ii). In both, we assume that there is a given list $\FF$ of pairwise non-overlapping $D$-simplices on $\gamma_D$ on exactly $n$ vertices which is non-extendable.

\begin{prop} \label{prop_countereg_extension_n}
There is a collection $\FF'$ of pairwise non-overlapping $D$-simplices on $\gamma_D$ on exactly $n+1$ vertices, $\FF\subseteq \FF'$, such that $\FF'$ is non-extendable.
\end{prop}
\begin{proof}
We consider the vertex set $[n]$ of the $D$-simplices of $\FF$ as a subset of $[n+1]$ (in $\gamma_D$). From the new vertex $n+1$, for every facet $\sigma$ of the cyclic polytope $C(n,D)$ which is visible from $n+1$
(check Subsection \ref{subsec_d=4} for the definition) we add the $D$-simplex $\sigma \cup \{n+1\}$ to $\FF$ in order to get an updated list $\FF'$. Certainly each of these new simplices in $\FF'$ does not overlap any of the given simplices, and to extend $\FF'$  to a triangulation of $C(n+1,D)$ we need to have a triangulation of $C(n,D)$ that extends $\FF$ which is impossible.
\end{proof}

For a triangulation $T \in S(n,d)$, let 
\[T/n=\{\sigma\subseteq [n]: \sigma \cup \{n\} \in T,\, n \notin \sigma\}.\]
We use the following known fact (see \cite[Theorem 4.2(ii)]{rambau}).

\begin{thm} \label{thm_link}
    Given $T \in S(n,d)$, $T/n$ is a triangulation in $S(n-1, d-1)$.
\end{thm}

\begin{prop} \label{prop_countereg_extension_D}
There is a collection $\FF'$ of pairwise non-overlapping $(D+1)$-simplices on $\gamma_{D+1}$ on exactly $n+1$ vertices which is non-extendable.
\end{prop}
\begin{proof}
 Let $\FF'=\FF*\{\{n+1\}\}$ and consider the elements of $\FF'$ as simplices on $\gamma_{D+1}$. Clearly any two simplices of $\FF'$ do not overlap in $\RR^{D+1}$, e.g. as they are not $(D+3)$-interlacing, see Proposition~\ref{prop_overlap}. We claim that $\FF'$ is non-extendable. Suppose otherwise that $\FF'$ can be extended to a triangulation $T \in S(n+1,D+1)$. By Theorem \ref{thm_link}, $T/(n+1)$ is a triangulation of $C(n,D)$, but it extends $\FF$, a contradiction.
\end{proof}

\begin{proof}[Proof of Theorem \ref{thm_main}(ii)]
Starting from the example of Proposition \ref{prop_eg_D=5_n=8}, we 
first construct an example in dimension $D$ with $D+3$ vertices by applying Proposition~\ref{prop_countereg_extension_D}, and then 
extend it to an example with $n$ vertices using Proposition \ref{prop_countereg_extension_n}.
\end{proof}
In~\cite[Prop.8.1]{opperman_thomas} non-extendable examples of 3 $D$-simplices on $D+3$ vertices on $\gamma_D$ are given for every \emph{even} $D\ge 6$ (explicitly for $D=6$ and implicitly for even $D\ge 8$);  our examples are different and work for every $D\ge 5$. 
We now show that in Theorem~\ref{thm_main}(ii) the lower bound $D+3$ on the vertex set size is tight.

\begin{prop} \label{prop_n<=D+2}
For each dimension $D$, given pairwise non-overlapping simplices on the moment curve $\gamma_D$ in $\RR^D$ on at most $D+2$ (and at least $D+1$) vertices  can be extended to a triangulation without adding new vertices.
\end{prop}
\begin{proof}
Let $\FF$ be a given list of non-overlapping simplices on $\gamma_D$ and let $n$ be the number of vertices. When $n=D+1$, we simply extend $\FF$ to the full simplex. 
When $n=D+2$, there is a unique (Radon) partition $[n]=\sigma \cup \tau$, so $\sigma$ and $\tau$ overlap in $\RR^D$.
So at least one of them, say $\sigma$, is not contained in any simplex of $\FF$. By this, every simplex of $\FF$ is a member of the triangulation $\partial \sigma*\langle\tau\rangle \in \HST(n,D)$, where $\partial \sigma=\{\sigma':\sigma'\subsetneq \sigma$\} and $\langle\tau\rangle=\{\tau':\tau'\subseteq \tau\}$.
\end{proof}

\begin{figure}[ht]
			\centering
\includegraphics[totalheight=5.5cm]{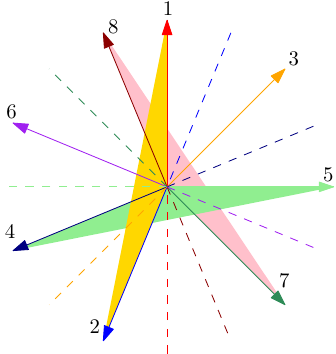}
    \caption{Description of the Gale transform of the construction of Proposition \ref{prop_eg_D=5_n=8}. The dual cones to $\sigma_1$, $\sigma_2$ and $\sigma_3$ are spanned by $\{7,8\}$, $\{1,2\}$ and $\{4,5\}$, respectively. Since they do not have intersection in the interior, the list is non-extendable.}
			\label{fig:d=5_n=8_gale}
		\end{figure}

\begin{remark}[Communicated with Francisco Santos]\label{remark_santos}
An alternative way to see non-extendability when $n=D+3$ can be provided via Gale transform, see Figure \ref{fig:d=5_n=8_gale}. It is known that when $n=D+3$ every triangulation over $n$ points in $\RR^D$ is regular. A given list of $D$-simplices is extendable to a regular triangulation without adding new vertices if and only if 
the intersection of all dual cones (the dual cone of $\sigma$ is the cone at the Gale dual over the complement $[n]\setminus \sigma$) has a full-dimensional intersection. Therefore, it is reduced to checking intersection of the dual cones in $\RR^2$. For more backgrounds, consult \cite{triangulations_book}.
\end{remark}

\section{Final remarks}\label{sec:final}
\textbf{Algorithmic aspects.} Theorem~\ref{thm_main}(i) raises the question of how fast the extension can be computed:

\begin{pro}\label{prob:running_time}
For $d=3,4$, given a geometric complex $\Delta$ on $n$-vertex set $V$ on the moment curve $\gamma_d$, find the running time of computing a triangulation $T \in S(V,d)$ of the cyclic $d$-polytope $\conv(V)$ which is an extension of $\Delta$, without adding new vertices; namely in the worst case, as $n$ tends to infinity. 
\end{pro}
Given such extension $T$, by coning the boundary of $T$ with a new vertex we obtain a triangulation of $S^d$, and thus, by the Dehn-Sommerville relations, $T$ has $O(n^2)$ faces, which is tight for some triangulations $T$; see also \cite[Thm.2.4]{opperman_thomas} for $d=4$. Thus the running time is $\Omega(n^2)$, the size of the output in the worst case.

\begin{prop}\label{prop:greedy_algorithm}
For both $d=3,4$, there exists a greedy algorithm that produces an extension $T \in S(V,d)$ of $\Delta$ as in Problem~\ref{prob:running_time} and runs in $O(n^5)$ time.   
\end{prop}
\begin{proof}
We may assume that $\Delta$ has dimension at most $d-1$, as for any $d$-simplex $\sigma \in \Delta$, every extension $T \in S(V,d)$ of $\Delta\setminus\{\sigma\}$ must contain $\sigma$ by Proposition \ref{prop_d/2-skeleton}.
We may also assume that the boundary complex $\partial \conv(V)$ is a subcomplex of $\Delta$ as $\partial \conv(V)$ is a subcomplex of any triangulation in $S(V,d)$. All these reductions only require $O_d(n^2)$ time.
Consider the following greedy algorithm: 
\begin{itemize}
    \item Step 0: Let $T_0=\Delta$. Make a list $L_0$ of all $(d-1)$-simplices in $\partial \conv(V)$. Then $L_0$ is non-empty and $T_0$ is a geometric complex on $\gamma_d$. This step requires $O(n^2)$ time.

\item Step $i$ (for $i\ge 1$):
\subitem If $L_{i-1}$ is empty, stop and declare $T=T_{i-1}$ is a desired extension of $\Delta$.

\subitem Else, choose a simplex $\tau\in L_{i-1}$. We will have either $\tau$ is in $\partial \conv(V)$ and is contained in no $d$-face of $T_{i-1}$, or $\tau$ is in $T_{i-1}\setminus \partial \conv(V)$ and is contained in exactly one $d$-face of $T_{i-1}$. 
Order by $\sigma_1,\ldots,\sigma_t$  the $d$-simplices on $V$ containing $\tau$ 
-- there are $O(n)$ of them -- and go through them checking if $\sigma=\sigma_j$ is $(d+2)$-interlacing 
with some facet of $T_{i-1}$. 

If it is not, let 
$T_i=T_{i-1}\cup \langle \sigma\rangle$ (where $\langle \sigma \rangle$ is the simplicial complex with a unique maximal face $\sigma$); by construction $T_i$ is a geometric simplicial complex on $\gamma_d$.
Now change $L_{i-1}$ as follows to obtain $L_i$: delete from $L_{i-1}$ the face $\tau$ and all $(d-1)$-faces of $\sigma$ which belong to $L_{i-1}$, and add to it all $(d-1)$-faces of $\sigma$ which do not belong to $L_{i-1}$. 

If $\sigma=\sigma_j$ does $(d+2)$-interlaces with some facet of $T_{i-1}$, let $\sigma=\sigma_{j+1}$ and repeat the check.

\end{itemize}

The scenario that each face $\sigma_j$ does $(d+2)$-interlace with some facet of $T_{i-1}$ is impossible by Theorem~\ref{thm_main}(i), hence the algorithm will stop when $L_i$ is empty, giving $T=T_{i}$ as a geometric simplicial complex which is an extension of $\Delta$. 
It is left to verify that $T$ triangulates $\conv(V)$.
Indeed, the simplices of $T$ cover the whole cyclic polytope $\conv(V)$:
if not, then there exists a point $p\in \conv(V)\setminus T$, and $p$ sees a $(d-1)$-face $\tau\in T$, hence $\tau\in L_i$, contradicting that $L_i$ is empty.

Running time analysis: a single check in Step $i$ whether a given simplex $\sigma$ is $(d+2)$-interlacing with any of the facets of $T_{i-1}$ or $\sigma \in T_{i-1}$ takes $O(n^2)$ time. 
If $\sigma$ is chosen to be added to $T_{i-1}$, it requires $O_d(n^2)$ more time to update $T_i$ (as repetition of faces are not allowed) and $O(n^2)$ time to update $L_i$.
Hence for a single $\sigma$ in Step $i$ we need $O(n^2)$ time.
This check is repeated in Step $i$ at most $O(n)$ times. The number of steps is at most $O(n^2)$ -- the total number of facets in $T$ --  hence  the total running time is $O_d(n^5)=O(n^5)$. 
\end{proof}

\smallskip
\textbf{Quantitative aspects.}
Theorem~\ref{thm_main}(ii) raises the question of how many new vertices suffice in order to extend:
\begin{pro}\label{prob:quantitative}
Let $d\ge 5$. Running over all geometric complexes $\Delta$ on an $n$-vertex set $V$ on the moment curve $\gamma_d$, find the minimum number $m(d,n)$ of new vertices that suffices for extension of $\Delta$ to a triangulation 
of $\conv(V)$, in the worst case.
\end{pro}

By Theorem~\ref{thm_main}(ii), for every $d\geq 5$ and $n\ge d+3$ we have $m(d,n)\ge 1$, but we were unable to find a better bound. There are studies on convex partitions of polyhedra \cite{chazelle_convex_partition, chazelle_palios_triangulating_nonconvex}, but in this context one may subdivide simplices of $\Delta$,
which we forbid in our setting.

\smallskip

\textbf{Connection to the lattice property.} Recall that when $d=2,3$,  Theorem \ref{lemma_lattice} guarantees that $\HST(n,d)$ is a lattice where the meet is given by submersion sets to satisfy 
\begin{align}
 \sub(T_1 \wedge T_2)=\sub(T_1) \cap \sub(T_2) \textrm{ for every $T_1,T_2\in S(n,d)$.} \label{condition_intersection}
\end{align}
 Condition (\ref{condition_intersection}) was important in our proof of Theorem \ref{thm_main}(i). 

However, imposing (\ref{condition_intersection}) for $\HST(n,d)$ is not necessary for extendability at one higher dimension $D=d+1$: While Proposition \ref{prop_n<=D+2} gives that given pairwise non-overlapping simplices on $\gamma_D$ on $n\leq D+2=d+3$ vertices are extendable to a triangulation in $S(n,D)$, it was shown 
that (\ref{condition_intersection}) is not satisfied when $d=D-1$ is $4$ or $5$ and $n=d+3=D+2$, see \cite[Section V, Remark 1]{edelman_cyclic_triangulation_envelope}. Still, $\HST(d+3,d)$ is a lattice (communicated from Nicholas J. Williams) and this can be seen, for instance, by the explicit description of this poset \cite[Lemma 3.1]{williams_twoordersequal}. However for $d=4, 5$, $\HST(d+5, d)$ is not a lattice \cite{edelman_rambau_reiner_subdivision_lattice_countereg, williams_stasheff_tamari_advances_lattice}.
See Table \ref{table:lattic_property} for a summary of known results.

\begin{table}[ht]
    \centering
    \begin{tabular}{|c||c|c|c|c|}
        \hline
         Conditions & $n=d+2$ & $n=d+3$ & $n=d+4$ & $n=d+5$ \\
         \hline
         \hline
         Condition (\ref{condition_intersection}) on  $\HST(n,d)$ & YES & NO \cite{edelman_cyclic_triangulation_envelope} & - & -\\
         \hline
         Extendability in $\RR^D=\RR^{d+1}$ & YES & YES & NO & NO\\
         \hline
        $\HST(n,d)$ being a lattice   & YES & YES \cite{williams_twoordersequal} & ? & NO \cite{edelman_rambau_reiner_subdivision_lattice_countereg, williams_stasheff_tamari_advances_lattice}\\
         \hline
    \end{tabular}
    \caption{Comparison between the lattice conditions and extendability when $d=4$ or $d=5$. Here $D=d+1$. The second row is from Theorem \ref{thm_main}(ii) and Proposition \ref{prop_n<=D+2}.}
    \label{table:lattic_property}
\end{table}

\begin{Q}
Is $\HST(d+4,d)$ a lattice for every $d \geq 4$?
\end{Q}

\smallskip

\textbf{Tightness for embedded $d$-simplices on $\gamma_d$.} 
Propositions~\ref{prop_eg_D=5_n=8} and~\ref{prop_countereg_extension_D} give a non-extendable example with respect to $d$-simplices that has three $d$-simplices on $\gamma_d$ (and on $d+3$ vertices), for every $d\ge 5$. 
No such examples with one $d$-simplex exist, even if the vertex set includes extra vertices from the moment curve. 
\begin{lemma}\label{lem:extend_one}
Let $d\ge 2$ and $\sigma\subseteq V\subseteq \gamma_d$ such that $V$ is finite and $|\sigma|=d+1$. Then $\{\sigma\}$ can be extended to a triangulation of $\conv(V)$ without adding vertices not in $V$.    
\end{lemma}
\begin{proof}
Order the vertices in $V\setminus \sigma$ in an arbitrary order $v_1,\ldots, v_l$. 
Let $T_0=\langle\sigma \rangle$, and 
$T_i=\cone(T_{i-1},v_i)$ for $i \in [l]$ (see Subsection \ref{subsec_d=4} for the definition). Then $T_i$ is a triangulation of $\conv(\sigma \cup \{v_1,\ldots v_i\})$, hence $T_l$ is a desired triangulation. 
\end{proof}
Lemma~\ref{lem:extend_one} answers  \cite[Question 5.4(1)]{novik2024transversalnumberssimplicialpolytopes} in the affirmative.    
\begin{remark}\label{rem:Rambau}
J\"org Rambau (personal communication) showed the following stronger statement, covering the case of two simplices as well, using Gale duality for its proof:
\begin{prop}[Rambau] 
Let $V$ be a finite set in convex position in $\RR^d$, and $\sigma$ and $\tau$ be two non-overlapping $d$-simplices with vertices in $V$. Then there exists a regular triangulation of $\conv(V)$ containing $\sigma$ and $\tau$ whose vertex set is $V$.
\end{prop}

\end{remark}
\section*{Acknowledgements}
The authors are grateful to Francisco Santos for his Remark \ref{remark_santos}, 
to Victor Reiner for inspiring comments on the lattice property of the higher Stasheff-Tamari poset, to Tillmann Miltzow for discussions on convex partitions of polyhedra, to J\"org Rambau for insightful comments around Proposition~\ref{prop_eg_D=5_n=8} and Remark~\ref{rem:Rambau}, 
and to G\"{u}nter M. Ziegler for helpful suggestions on the writing.
Our special thanks go to Nicholas J. Williams for many  helpful suggestions and for explaining to us the algebraic implication Corollary~\ref{cor:algebraic}.

\bibliographystyle{hplain}
\bibliography{bibliography}

\end{document}